\documentclass[11pt]{article}
\usepackage[top=1in, bottom=1in, left=1.25in, right=1.25in]{geometry}

\usepackage{amsmath,amssymb,amsthm,xcolor,enumerate,bbm}
\usepackage[unicode,breaklinks=true,colorlinks=true]{hyperref}

\usepackage{theoremref}
\usepackage{mathrsfs}
\usepackage[]{authblk}
\usepackage{comment}
\usepackage{cancel}
\usepackage{graphicx}
\numberwithin{equation}{section}
\newtheorem{thm}{Theorem}[section]

\newtheorem{lem}[thm]{Lemma}
\newtheorem{prop}[thm]{Proposition}

\theoremstyle{remark}
\newtheorem{remark}[thm]{Remark}
\theoremstyle{definition}

\newcommand{\bke}[1]{\left ( #1 \right )}
\newcommand{\bkt}[1]{\left [ #1 \right ]}
\newcommand{\bket}[1]{\left \{ #1 \right \}}
\newcommand{\norm}[1]{\left \| #1 \right \|}

\newcommand{\R}{\mathbb{R}}
\newcommand{\N}{\mathbb{N}}
\renewcommand{\div}{\mathop{\rm div}\nolimits}

\newcommand{\sgn}{\mathop{\rm sgn}\nolimits}
\newcommand{\pd}{\partial}
\newcommand\Ga{\Gamma}
\newcommand\ga{\gamma}

\newcommand{\si}{\sigma}
\newcommand\De{\Delta}
\newcommand\de{\delta}
\newcommand{\nb}{\nabla}
\newcommand{\lec}{{\ \lesssim \ }}
\newcommand{\gec}{{\ \gtrsim \ }}

\newcommand{\bka}[1]{{\langle #1 \rangle}}
\newcommand{\abs}[1]{\left | #1 \right |}
\newcommand{\mat}[1]{\begin{bmatrix} #1 \end{bmatrix}}
\newcommand\al{\alpha}
\newcommand\be{\beta}

\newcommand\e {\varepsilon}  %

\renewcommand\th{\theta}

\newcommand\Si{\Sigma}

\newcommand{\LN}  {\mathrm{L\mkern-0.5mu N}}

\newcommand{\esssup} {\mathop{\mathrm{ess\,sup}}}
\newcommand{\supp} {\mathop{\mathrm{supp}}}

\newcommand{\na}{\nabla}
\newcommand{\td}{\tilde}
\newcommand{\wt}{\widetilde}

\newcommand{\Eq}[1]{\begin{equation*} #1 \end{equation*}}
\newcommand{\EQ}[1]{\begin{equation} #1 \end{equation}}
\newcommand{\EQS}[1]{\begin{equation}\begin{split} #1 \end{split}\end{equation}}

\newcommand{\EQN}[1]{\begin{equation*}\begin{split} #1 \end{split}\end{equation*}}
\newcommand{\EN}[1]{\begin{enumerate} #1 \end{enumerate}}

\newcommand{\lot}{{\mathrm{l.o.t.}}}

\newcommand{\one}{\mathbbm{1}}
\begin{document}

\title
{Finite energy Navier-Stokes flows with unbounded gradients induced by localized flux in the half-space}

\author[1]{\rm Kyungkeun Kang}
\author[2]{\rm Baishun Lai}
\author[3]{\rm Chen-Chih Lai}
\author[4]{\rm Tai-Peng Tsai}
\affil[1]{\footnotesize Department of Mathematics, Yonsei University, Seoul 120-749, South Korea}
\affil[2]{\footnotesize LCSM (MOE) and School of Mathematics and Statistics, Hunan Normal University, Changsha 410081, Hunan, China}
\affil[3,4]{\footnotesize Department of Mathematics, University of British Columbia, Vancouver, BC V6T 1Z2, Canada}

\date{}
\maketitle

\begin{abstract}

For the Stokes system in the half space, Kang [Math.~Ann.~2005] showed that a solution generated by a compactly supported, H\"older continuous boundary flux may have unbounded normal derivatives near the boundary. In this paper we first prove explicit global pointwise estimates of the above solution, showing in particular that it has finite global energy and its derivatives blow up everywhere on the boundary away from the flux. We then use the above solution as a profile to construct solutions of the Navier-Stokes equations which also have finite global energy and unbounded normal derivatives due to the flux. Our main tool is the pointwise estimates of the Green tensor of the Stokes system proved by us in \cite{Green} (arXiv:2011.00134). %
We also examine the Stokes flows generated by dipole bumps boundary flux, and identify the regions where the normal derivatives of the solutions tend to positive or negative infinity near the boundary.

\medskip


\medskip

\end{abstract}

\renewcommand{\baselinestretch}{0.8}\normalsize
\tableofcontents
\renewcommand{\baselinestretch}{1.0}\normalsize

\section{Introduction}

Denote $x=(x_1,\ldots,x_{n-1},x_n) = (x',x_n)$ for $x \in \R^n$, $n \ge 2$,
 $\R^n_+ = \{  (x',x_n)\in \R^n \ | \ x_n>0\}$, and $\Si=\pd\R^n_+$.
The nonstationary Stokes system in the half-space $\R_{+}^{n}$, $n\geq2$, reads
\begin{align}\label{E1.1}
\begin{split}
\left.
\begin{aligned}
u_{t}-\Delta u+\nabla \pi=f \\
 \div u=0
\end{aligned}\ \right\}\ \ \mbox{in}\ \ \R_{+}^{n}\times (0,\infty),
\end{split}
\end{align}
with initial and boundary conditions
\begin{align}\label{E1.2}
u(\cdot,0)=u_0; \qquad  %
u(x',0,t)=\phi(x',t)\ \ \mbox{on}\ \ \Si\times (0,\infty).
\end{align}
Here  $u=(u_{1},\ldots,u_{n})$ is the velocity, $\pi$ is the pressure, and $f=(f_{1},\ldots,f_{n})$ is the  external force. They are defined for $(x,t)\in \R_{+}^{n}\times (0,\infty)$.

It was shown by the first author in \cite{Kang2005} an example of a H\"older continuous Stokes flow with \emph{unbounded normal derivatives} of the tangential components of the velocity in a region $E$ near the boundary.
It is caused by a nonzero boundary flux away from $E$.  The paper \cite{Kang2005} only considered the local behavior of the solution. Our first goal in this paper is to find explicit global pointwise bounds of the solution and its derivatives, and show that it has \emph{finite global energy}.  We will also show explicit lower bounds of  the normal derivative of its tangential components, showing their \emph{everywhere} blow up on the boundary away from the flux.
Our second goal is to use it as a profile to construct solutions of the Navier-Stokes equations
\EQ{\label{NS}
u_{t}-\Delta u+\nabla \pi= - u \cdot \nb u,\quad
 \div u=0,\quad
 \mbox{in}\ \ \R_{+}^{n}\times (0,\infty),
}
with similar properties.

Specifically,
the solution in \cite{Kang2005} is constructed from a H\"older continuous boundary data $\phi(x',t)$ of compact support using the Golovkin tensor $K_{ij}$, i.e., the Poisson kernel for  \eqref{E1.1} in the half space $\R^n_+$, see \eqref{Golovkin}.
The boundary data is of the form
\EQ{\label{boundary-data1}
\phi(\xi',s) = g(\xi') \,h(s)\, e_n,\quad e_n = (0,\ldots,0,1),
}
where
\EQ{\label{boundary-data2}
g \in C^3_c(B_1'), \quad
h\ge 0\in C_c(\R) \cap C^1(0,1),
}
and
\EQ{\label{boundary-data3}
 \supp h \subset [\tfrac14,1],\quad h(s) = (1-s)^a \text{ for }s\in [\tfrac12,1],  \quad 0<a\le\tfrac12.
}
It has only the normal component because of the unit normal vector $e_n$.
The exponent $a$ is equal to $\tfrac12$ in \cite{Kang2005}.
We will also consider $0<a<\tfrac12$ in this paper. Note that
$h(t)$ is H\"older continuous of exponent $a$, and its lack of further regularity in time is the cause for unbounded derivatives.
We do not assume as in \cite{Kang2005} that $g$ is radial and positive. They are unnecessary to show the \emph{upper bounds} in Proposition \ref{th3.1}.  For the \emph{lower bounds}, in fact, we will choose a nonradial $g$ in a product form in \eqref{gxiform} for more accurate estimates in Theorem \ref{th1.1} and Propositions \ref{th1.2}, and choose $g$ as a sum of two bumps of opposite signs in the Appendix A (\S\ref{Sec5}).

Denote by $\hat v(x,t)$ the solution of the Stokes system \eqref{E1.1} generated by the boundary data
\eqref{boundary-data1}--\eqref{boundary-data3} using the Golovkin tensor, with zero initial data and zero force; See \eqref{hatv.formula} for its formula.
It is showed in \cite{Kang2005} that, for $n=3$, $a=\tfrac12$ and radial $g(\xi')$,  for any $x_0' \in \Si$ with $|x_0'|>10$, we have that $\hat v\in C^{2b,b}(\overline{B_{x_0',r}}\times [0,2])$ for any $b$ in $(0,1/2)$, $0<r<1$, and at time $t=1$, for each $i=1,2$,
\EQ{\label{eq1.7}
\lim_{x_3 \to 0_+} |\pd_3 \hat v_i(x_{1},x_{2}, x_3,1) |= \infty, \quad \text{if}\quad x_{i}\not =0, \quad |(x_1,x_2)|>10.
}

A bounded but not H\"older continuous solution with unbounded normal derivative on $\Si$ at $t=1$
is also obtained in \cite[Remark 6]{Kang2005} by choosing $h(s) = \sum_{k=2}^\infty a_k(1-s)^{1/k}$ for $s\in [\frac12,1]$, with $a_k \ge 0$ and $\sum_{k=2}^\infty k^2 a_k<\infty$.
Its corresponding pressure at $t=1$ is in $L^q(B_{x_0',1}'\times (0,2))$ for $q=1$ but not for $q>1$. It shows that sufficient integrability of the pressure is necessary for the H\"older continuity of the velocity.

A more singular (very weak) solution of the Stokes system, also generated by boundary flux, was constructed in Chang and the first author \cite[Proposition 3.2]{C-K-20} for $n \ge 2$
so that $\norm{\nabla u}_{L^2_{t,x}(B_1^+ \times (0,1))}=\infty$,  although $\norm{u}_{L^4_tL^p_x(\R^n_+\times (0, 1))}$ is bounded for all $p>n/(n-1)$, with compactly supported boundary data belonging to $L^4_tL^{\infty}_x(\Si\times (0,1))$.
It is used as the profile to construct non-smooth solutions near the boundary of the Navier-Stokes equations in \cite[Section 4.2]{C-K-20}. This construction is possible because it only involves $L^q$-type estimates.

Seregin and \v Sver\'ak \cite{SeSv10}
have an example of a bounded shear flow that solves both Stokes and Navier-Stokes equations, has zero initial and boundary values, and has unbounded normal derivatives. However, this example has no spatial decay as it is a shear flow. Another viewpoint is that it has ``nonzero boundary value at spatial infinity''.
Hence it has \emph{infinite} global energy.
It seems very difficult to construct a solution with similar properties but with \emph{finite} global energy.

These examples are motivated by the regularity problem near the boundary for the Stokes system \eqref{E1.1} and the Navier-Stokes equations \eqref{NS}. They show that certain properties that are valid in the interior or globally may fail on the boundary. {In fact, in the interior case, one usually defines a solution of \eqref{NS} to be \emph{regular} at $(x_0,t_0)$ if it is bounded or H\"older continuous in $B_r(x_0)\times (t_0-r^2,t_0)$ for some $r>0$, since which imply the continuity of all higher spatial derivatives. This property fails in the boundary case, as shown by the examples of \cite{Kang2005,SeSv10}.}
For the theory of boundary regularity, see for example Seregin \cite{MR1810618, MR1891072, MR2749216,MR3509681}, Kang-Gustafson-Tsai \cite{GKT}, Mikhailov \cite{MR2749212,MR2749371} and Dong-Gu \cite{MR3255469} for regularity criteria,
Chernobay \cite{MR3904061} and Seregin \cite{MR4105900} for type I blowups near the boundary,
Chang-Jin \cite{MR3398794}  for H\"older continuity,
and Chang-Choe-Kang \cite{MR3815542} for Dini continuity. See also the survey paper \cite{MR3287793} of Seregin and Shilkin.

\medskip

The following is our first main result.

\begin{thm}
\label{th1.1}
Let $n \ge 2$ and $\phi(\xi',s)$ be defined on $\Si \times \R$ by
\eqref{boundary-data1}-\eqref{boundary-data3} with $0<a \le 1/2$.
Let $(\hat v(x,t),\hat p(x,t))$ be the solution of the Stokes system \eqref{E1.1} with boundary data $\phi$, zero initial data and zero force, {obtained using the Golovkin tensor in \eqref{hatv.formula} and \eqref{hatp.formula}.}
Then $\hat v$ satisfies the pointwise bounds in Proposition \ref{th3.1}
and is H\"older continuous, $\hat v\in C^{2b,b}(\overline{\R^n_+}\times [0,2])$ where $b=a$ if $0<a<1/2$, and $b$ can be any number in $(0,1/2)$ when $a=1/2$.
 In particular, when $n\ge 3$, it has finite global energy,
\EQ{\label{th1.1eq1}
\esssup _{0<t<2} \int_{\R^n_+} |\hat v(x,t)|^2\,dx + \int_0^2 \int_{\R^n_+} |\nb\hat v(x,t)|^2\,dx\,dt<\infty.
}

Moreover, for $n\ge2$ we have the following pointwise estimate for the pressure $\hat p$
\EQ{\label{est-p-hat}
|\hat p(x,t)|
\le C \frac{\chi_{\frac14\le t\le1}(t) }{(1-t)^{1-a}} \bke{\one_{n\ge3}\, \frac1{\bka{x}^{n-2}} + \one_{n=2} \log(|x|+2)} + C \frac{\chi_{t>1}(t)}{|t-1|^{1-a}\bka{x}^{n-1}}
}
for some constant $C>0$ independent of $x$ and $t$.

However, when $n \ge 3$, if we choose $g(\xi')$ in \eqref{boundary-data1} of the product form
\EQ{\label{gxiform}
g(\xi')=\textstyle \prod_{j=1}^{n-1} \mathcal{G}(\xi_j)
}
where $\mathcal{G}:\R\to\R$ is smooth, even, supported in $(-\frac 4{5\sqrt{n-1}},\frac4{5\sqrt{n-1}})$,  $\mathcal{G}(\zeta)=1$ for $|\zeta|<\frac 1{2\sqrt{n-1}}$, and $\mathcal{G}'(\zeta)\le 0$ for $\zeta >0$, then
for $x=(x_1,\ldots,x_n)\in\R^n_+$,
\[
|x'|\ge3,\quad
x_n\le 1,
\]
and if $i<n$ satisfies $|x_i| = \max_{1\le j \le n-1} |x_j|$,
we have
\EQ{\label{th1.1eq2}
| \pd_{x_n} \hat v_i(x,1) | \ge
\left\{\begin{array}{ll}
\displaystyle
\frac{C_1} {|x'| ^{n-1}}\log \frac 2{x_n} - \frac{C_2} {|x'| ^{n-2}}
& \text{ if } a=\tfrac12,\\[10pt] %
\displaystyle
\frac{ C_1 } {|x'| ^{n-1}\, x_n^{1-2a}}- \frac{C_2} {|x'| ^{n-2}}
& \text{ if } 0<a<\tfrac12,\end{array}\right.
}
where $C_1$ and $C_2$ are positive constants independent of $x$.
\end{thm}

\emph{Comments on Theorem \ref{th1.1}.}
\EN{
\item Theorem \ref{th1.1} improves \cite{Kang2005} in the sense that the upper and lower bounds are explicit, global, and with spatial decay. They imply finite global energy. Theorem \ref{th1.1} also treat the case $0<a<\frac12$. See Comment 2 on Proposition \ref{th1.2} for more details.

\item Eq.~\eqref{th1.1eq2} implies \emph{everywhere blow-up} of the normal derivative $\pd_{x_n} \hat v'$ of the tangential velocity  $\hat v'= (\hat v_1,\ldots, \hat v_{n-1})$ away from the flux,
\EQ{\label{th1.1eq2b}
\lim_{x_n \to 0_+} \max_{i<n} | \pd_{x_n} \hat v_i(x',x_n,1) | =\infty, \quad \forall x' \in \Si, \ |x'|\ge 3.
}

\item The lower bound \eqref{th1.1eq2} is restricted to $n \ge 3$ so that the key Lemma \ref{th3.2} is valid.

\item The Stokes system is a linear flow and allows superposition. Hence if we use boundary data $\td \phi(x',s) = \phi(x',s)  + \phi(x',s-2) $, the solution has unbounded derivatives at times $t=1$ and $t=3$. This in turn can be used to construct solutions of Navier-Stokes equations with unbounded derivatives at times $t=1$ and $t=3$. We do not state this as a formal result since this is much simpler than multiple singularity results whose singularities are due to nonlinearities such as nonlinear  Schr\"odinger equations \cite{MR1655515} and nonlinear heat equations \cite{MR2115464}.

\item The choice \eqref{gxiform} of boundary data $g(\xi')$ is supported in $|\xi'|\le1$ with a constant sign. We call this a \emph{single bump}. We will also consider \emph{dipole bumps} in the Appendix A (\S\ref{Sec5})
where two single bumps with opposite signs are combined.
}

{Proposition \ref{th1.2} below gives a sharpened version of \eqref{th1.1eq2} of Theorem \ref{th1.1}:}
The lower bound \eqref{th1.1eq2} is valid when $i<n$ satisfies $|x_i| = \max_{1\le j \le n-1} |x_j|$. In Proposition \ref{th1.2} we remove this condition and only require $x_i\not = 0$. Hence the constants are larger and also depend on $|x_i|$. Note that we also specify the sign of $\pd_{x_n} \hat v_i(x,1) $.

\begin{prop}\label{th1.2}
Let $n \ge 3$. For the same solution $\hat v$ in Theorem \ref{th1.1} with the same choice of $g(\xi')$ in \eqref{gxiform},
for every $i<n$, $\pd_n \hat v_i (x',x_n)$ blows up as $x_n \to 0_+$ for any $x' \in \Si$ with $x_i \not =0$ and $|x'|>3$, without assuming $|x_i| = \max_{1\le j \le n-1} |x_j|$. We have estimates of the form
\EQ{\label{th1.1eq3}
-\sgn(x_i) \, \pd_{x_n} \hat v_i(x,1)  \gec
\left\{\begin{array}{ll}
\displaystyle
\frac{|x_i|}{|x'|^{n}}\,  \log \frac 2{x_n} - \frac{C} {|x'|^{n-2 }}  \log \frac{4|x'|}{|x_i|}
& \text{ if } a=\tfrac12,\\[10pt]
\displaystyle
\frac{|x_i|}{|x'|^{n}}\,  \frac{ 1} { x_n^{1-2a}}- \frac{C}{|x_i|^{1-2a} |x'|^{n-3+2a}}
& \text{ if } 0<a<\tfrac12,\end{array}\right.
}
where C is a positive constant independent of $x$,
for all $|x'|>3$, $0<x_n<1$ and $x_i \not =0$.
We also have $\pd_{x_n} \hat v_i(x,1)=0$ if $x_i=0$.
\end{prop}

\emph{Comments on Proposition \ref{th1.2}.}
\EN{

\item {Extending \eqref{th1.1eq2b},} \eqref{th1.1eq3} implies everywhere blowup on the boundary away from the flux {of the normal derivative of \emph{each tangential component} of the velocity,}
in the sense that, for every $i<n$,
\EQ{
\limsup_{x \to (x'_0,0)} |\pd_{x_n} \hat v_i(x,1)| = \infty, \quad \forall x'_0 \in \Si, \ |x'_0|>3,
}
even if the $i$-th component of $x'_0$ is $0$. It is already known in \cite{Kang2005} although not explicitly stated. It follows from \eqref{eq1.7}.

\item Eq.~\eqref{th1.1eq2} of Theorem \ref{th1.1} and \eqref{th1.1eq3} of Proposition \ref{th1.2} improve \cite{Kang2005} because of their uniform estimates of the main term and the error term. In the corresponding estimates of \cite{Kang2005}, the coefficient of the leading singular term has implicit dependence on $x_0'$, and the error term is shown to be bounded but its decay is not shown.  Also, although \cite[Remark 5]{Kang2005} mentioned the case $0<a<1/2$, it does not provide any details. Indeed, in our proof, this case has two competing leading singular terms and is more delicate.
The choice of the product form \eqref{gxiform} of $g(\xi')$ instead of being radial makes these explicit bounds possible.

\item As $x_n \to 0_+$, $\pd_{x_n} \hat v_i(x,1)$ converges to $\infty$ in the region $x_i<0$, and to $-\infty$ in the region $x_i>0$. Each of these regions is connected. It is different from the dipole bumps case, to be considered in Proposition \ref{thA.1}.

}

Our next theorem asserts the existence of a solution $u(x,t)$ of the Navier-Stokes equations \eqref{NS} with properties similar to the Stokes flow considered in Theorem \ref{th1.1}.

\begin{thm}
\label{th1.3}
Let $n \ge 3$ and $\phi(\xi',s)$ be defined on $\Si \times \R$ by
\eqref{boundary-data1}-\eqref{boundary-data3} with $0<a \le 1/2$. There is a small constant $\al_0>0$ such that, for any $\al \in (-\al_0,\al_0)$,
there is a solution $(u,\pi)$ of the Navier-Stokes equations \eqref{NS} with boundary data $\al\phi$, zero initial data and zero force. The solution $u$ satisfies the pointwise bounds \eqref{eq-0415-a}, has finite global energy \eqref{th1.1eq1}, and is H\"older continuous, $u\in C^{2b,b}(\overline{\R^n_+}\times [0,2])$ where $b=a$ if $0<a<1/2$, and $b$ can be any number in $(0,1/2)$ when $a=1/2$.
Moreover, the pressure $\pi$ satisfies
\EQ{\label{pi-regular}
   \pi \in L^r(0,2; L^m(\R^n_+)), \quad   1<\forall r< \frac 1{1-a},  \quad \frac n{n-2}  <\forall m < \infty.
}

However, if we choose $g(\xi')$ in \eqref{boundary-data1} of the form \eqref{gxiform}, then
for $x=(x_1,\ldots,x_n)\in\R^n_+$,
\[
|x'|\ge3,\quad
x_n\le 1,
\]
and if $i<n$ satisfies $|x_i| = \max_{1\le j \le n-1} |x_j|$,
we have
\EQ{\label{th1.3eq1}
| \pd_{x_n} u_i(x,1) | \ge
\left\{\begin{array}{ll}
\displaystyle
\frac{C_1|\al|} {|x'| ^{n-1}}\log \frac 2{x_n} - \frac{C_2 |\al| } {|x'| ^{n-2}}
& \text{ if } a=\tfrac12,\\[10pt] %
\displaystyle
\frac{ C_1 |\al|} {|x'| ^{n-1}\, x_n^{1-2a}}- \frac{C_2|\al|} {|x'| ^{n-2}}
& \text{ if } 0<a<\tfrac12,\end{array}\right.
}
where $C_1$ and $C_2$ are positive constants independent of $x$.
\end{thm}

\emph{Comments on Theorem \ref{th1.3}.}
\EN{
\item Theorem \ref{th1.3} gives solutions with pointwise upper bounds for themselves and their derivatives, and pointwise lower bounds for their normal derivatives at time $t=1$. Moreover, these solutions have finite global energy and are H\"older continuous up to boundary.

\item To obtain pointwise bounds for the contribution from the nonlinear term, we use the pointwise estimates of the Green tensor of the Stokes system in the half space, proved in \cite{Green}.

\item The condition $n\ge3$ is used in two places in the proof of Theorem \ref{th1.3}: First, it is needed for the lower bound \eqref{th1.1eq2} of Theorem \ref{th1.1} for the profile $\hat v$. Second,  the pointwise estimates of the Green tensor
are simpler when $n \ge 3$; Compare \thref{thm3} with \cite[Theorem 1.5]{Green}.

\item In most known boundary regularity criteria {(e.g.~\cite{MR1810618, MR1891072,MR2749216,MR3509681,GKT,MR2749212,MR2749371,MR3255469})} for Navier-Stokes equations,
certain integrability conditions on the pressure
are needed
to at least exclude the parasitic solutions of Serrin. {See \cite{MR1151263, MR2030374,MR2250016} for criteria involving only the pressure.}
Our pressure estimate \eqref{pi-regular} suggests that any weaker condition on the pressure is probably insufficient {for the continuity of velocity derivatives.}
}

As mentioned at the beginning, we will use the solution $\hat v$ of Theorem \ref{th1.1} as the profile, and construct the solution $u$ in Theorem \ref{th1.3} by Picard iteration. The key to the construction is the pointwise bound of the (unrestricted) Green tensor for \eqref{E1.1} proved by us
in \cite{Green}.  After we have pointwise {\emph{upper and lower}} bounds for the linear solution $\hat v$, we use the pointwise bound of the Green tensor to {get an \emph{upper bound} for the difference $v=u - \al \hat v$ generated by the nonlinearity $u\cdot \nb u $.} Pointwise bounded solutions of the Navier-Stokes equations in the half space have been constructed by various authors, see the review of \cite{Green} after \cite[(1.24)]{Green}. However, their constructions rely on explicit expansion of the nonlinearity with Helmholtz projection applied, which seem complicated to get localized estimates. Theorem \ref{th1.3} is in fact a major motivation for the paper \cite{Green}.

In Appendix A (\S\ref{Sec5}) we study the boundary blow up of the Stokes flows generated by dipole bumps boundary flux of the form \eqref{boundary-data1}--\eqref{boundary-data3} in $\R^3_+$ with
\EQ{\label{dipole-form}
g(\xi_1,\xi_2) = \mathcal{G}(\xi_1+10)\, \mathcal{G}(\xi_2) - \mathcal{G}(\xi_1-10)\, \mathcal{G}(\xi_2).
}
Our results are summarized in Proposition \ref{thA.1}.
Unlike \eqref{gxiform} where $g\ge0$, \eqref{dipole-form} consists of two bumps with opposite signs.
Thus, it gives an in-flux for $\xi_1 \sim -10$, and an out-flux for $\xi_1 \sim 10$. We will identify the regions where the normal derivatives of the solutions tend to positive or negative infinity near the boundary. These regions are disconnected and appear to be bordered by curves with asymptotic lines.

The rest of this paper is organized as follows.
In Section \ref{Sec2} we recall some preliminary results including integral estimates, the Golovkin tensor, the Green tensor, and a few other functions.
In Section \ref{Sec3} we prove pointwise estimates for the Stokes flow of \cite{Kang2005} and its derivatives. We also prove lower bounds of its derivatives that implies blow up. These are summarized in Theorem \ref{th1.1} and Proposition \ref{th1.2}.
In Section \ref{Sec4} we use the previous solution as the profile to construct solutions of the Navier-Stokes equations with the same pointwise estimates and upper and lower bounds of derivatives. It gives Theorem \ref{th1.3}.
In Appendix A (\S\ref{Sec5}) we study the boundary blow up of the Stokes flows generated by dipole bumps boundary flux.
In Appendix B (\S\ref{Sec6}) we prove derivative formulas and estimates of the function $C_i(x,y,t)$ to be defined in \eqref{eq_def_Ci} and its derivatives.

\section{Preliminaries}\label{Sec2}

In this section we recall some preliminary results.
We first recall a few useful integral estimates. We then recall the Golovkin tensor and the Green tensor, and a few other functions useful for their study.

\begin{lem}\label{lemma2.1}
For positive $L,a,d$, and $k$ we have
\[\int_0^L\frac{r^{d-1}\,dr}{(r+a)^k}\lesssim\left\{\begin{array}{ll}L^d(a+L)^{-k}& \text{ if } k<d,\\[3pt] L^d(a+L)^{-d}(1+\log_+\frac{L}a) & \text{ if } k=d,\\[3pt] %
L^d(a+L)^{-d}a^{-(k-d)}
& \text{ if } k>d.\end{array}\right.\]
\end{lem}
This is \cite[Lemma 2.1]{Green}.

\begin{lem}\label{lemma2.2}
Let $a>0$, $b>0$, $k>0$, $m>0$ and $k+m>d$. Let $0\not=x\in \R^d$ and
\[I:=\int_{\R^d}\frac{dz}{(|z|+a)^k(|z-x|+b)^m}.\]
Then, with $R=\max\{|x|,\,a,\,b\}\sim|x|+a+b$,
\EQN{I\lesssim R^{d-k-m} + \de_{kd} R^{-m} \log \frac Ra
+ \de_{md} R^{-k} \log \frac Rb
+ \mathbbm 1_{k>d} R^{-m}a^{d-k}
+ \mathbbm 1_{m>d} R^{-k}b^{d-m}.
}
\end{lem}
This is \cite[Lemma 2.2]{Green}.

\begin{lem}
\label{th2.3}
For $m\ge 1$ and $0<b<10$,
\begin{align}
\label{0223a}
\int_0^{1} \frac {-\log u}{(b+u)^m} du &\lec  1+  b^{1-m} |\log b|
+\de_{m1} (\log b) ^2.
\end{align}
\end{lem}
\begin{proof}
It is clear if $b\ge 1/2$ and
we may assume $0<b<1/2$.
If  $m>1$,
\[
\int_0^{1} \frac {-\log u}{(b+u)^m} du
= \int_0^b + \int_b^{1}
\le b^{-m} \int_0^b (-\log u) du+ \int_b^{1}  \frac {-\log u}{u^m} du
\lec b^{1-m} |\log b|.
\]
If $m=1$,
as $\int \frac{\log u}u du = \tfrac12 (\log u)^2$,
\[
\int_0^{1} \frac {-\log u}{b+u} du
\lec   (\log b) ^2.
\]
Combined, we get \eqref{0223a}.
\end{proof}

We next recall a few functions.
The \emph{heat kernel} $\Ga$ and the \emph{fundamental solution} $E$ of $-\De$ are given by
$$
\Gamma(x,t)=\left\{\begin{array}{ll}(4\pi t)^{-\frac{n}{2}}e^{\frac{-x^{2}}{4t}}&\ \text{ for }t>0,\\[3pt] 0&\ \text{ for }t\le0,\end{array}\right.\ \text{ and }\ E(x)=\left\{\begin{array}{ll}\frac1{n\,(n-2)|B_1|}\,\frac1{|x|^{n-2}}&\ \text{ for }n\ge3,\\[3pt] -\frac1{2\pi}\,\log|x|&\ \text{ if }n=2.\end{array}\right.
$$

The \emph{Golovkin tensor} $K_{ij}(x,t): \R_{+}^{n}\times \R\to \R$  is the Poisson kernel of the nonstationary Stokes system \eqref{E1.1} in the half-space $\R^n_+$, first constructed by Golovkin
\cite{MR0134083} for $\R^3_+$. A solution {$(\hat v(x,t),\hat p(x,t))$} of \eqref{E1.1} with zero force, zero initial value and boundary value
\begin{equation}\label{0903a}
\hat{v}(x',0,t)=\phi(x',t), \ \ \mbox{on}\ \ \Si\times (0,\infty),
\end{equation}
is given by {(see \cite[(84)-(85)]{MR0171094} and \cite[(2.17)-(2.19)]{Green})}
\EQ{\label{Golovkin}
\hat{v}_i(x,t)=\sum_{j=1}^n\int_{-\infty}^{\infty}\int_\Si K_{ij}(x-\xi',t-s)\phi_j(\xi',s)\,d\xi'\,ds,
}
{
\EQS{\label{eq-hat-p-def}
\hat p(x,t)=&~2\sum_{i=1}^n\pd_i\pd_n\int_{\Si}E(x-\xi')\phi_i(\xi',t)\,d\xi'
+2\int_{\Si}E(x-\xi')\pd_t\phi_n(\xi',t)\,d\xi'\\
&-4\sum_{i=1}^n(\pd_t-\De_{x'})\int_{-\infty}^\infty\int_{\Si}\pd_iA(x-\xi',t-\tau)\phi_i(\xi',\tau)\,d\xi'd\tau,
}
}
where
we extend $\phi(x',t)=0$ for $t<0$ and
the Golovkin tensor $K_{ij}(x,t)$ is
\EQS{\label{eq_def_Kij}
K_{ij}(x,t)&=-2\,\de_{ij}\,\pd_n\Ga(x,t)-4\,\pd_j\int_0^{x_n}\int_\Si\pd_n\Ga(z,t)\,\pd_iE(x-z)\,dz'\,dz_n\\
&\quad -2\,\de_{nj}\pd_iE(x)\de(t).
}
For $n \ge 2$, the Golovkin tensor satisfies,
for $i,j=1,\ldots,n$ and $t>0$,
\begin{equation}\label{eq_estKij}
\left|\pd_{x'}^l\pd_{x_n}^k\pd_t^mK_{ij}(x,t)\right|\lesssim\frac1{t^{m+\frac12}\left(x^2+t\right)^{\frac{l+n-\si}2}(x_n^2+t)^{\frac{k+\si}2}}, \quad \si=\de_{i<n} \de_{jn}.
\end{equation}
Here $\si=1$ if $i<n=j$ and $\si=0$ otherwise. See \cite{MR0171094} and \cite[(2.20)]{Green}.

The \emph{Green tensor} $ G_{ij}(x,y,t)$ of the Stokes system \eqref{E1.1}
is defined for $(x,y,t) \in\R^n _+ \times \R^n_+ \times \R$ and
 $1\le i,j\le n$ so that,
for suitable $f$ and $u_0$ and zero boundary value $\phi=0$,
the solution of \eqref{E1.1} is given by
\begin{equation}
\label{E1.3}
u_i(x,t) = \sum_{j=1}^n\int_{\R^n_+}G_{ij}(x,y,t)(u_0)_j(y)\,dy+
\sum_{j=1}^n\int_0^t \int_{\R^n_+} G_{ij}(x,y,t-s) f_j(y,s)\,dy\,ds.
\end{equation}

The following theorem is \cite[Theorem 1.5]{Green} restricted to dimension $n \ge 3$.

\begin{thm}[Green tensor estimates]\thlabel{thm3}
Let $n\ge 3$, $x,y\in\R^n_+$, $t>0$, $i,j=1,\ldots,n$, and $l,k,q,m \in \N_0$.
We have
\EQS{\label{eq_Green_estimate}
|\pd_{x',y'}^l \pd_{x_n}^k \pd_{y_n}^q \pd_t^m G_{ij}(x,y,t)|\lec&\frac1{(|x-y|^2+t)^{\frac{l+k+q+n}2+m}}
\\&+\frac{1}
{t^{m}(|x^*-y|^2+t)^{\frac{l+n+\si_{ijkq}}2}((x_n+y_n)^2+t)^{\frac{k+q-\si_{ijkq}}2} },
}
where $
\si_{ijkq} = (\de_{in} + \de_{jn})(1-\de_{k0}\de_{q0}) - \de_{in} \de_{jn} \de_{k+q=1}\in \{0,1,2\}$, and
$k+q-\si_{ijkq}=(k+q-(\de_{in}+\de_{jn}))_+$.
\end{thm}

We will use the following functions
\begin{equation}\label{eq_def_A}A(x,t)=\int_\Si\Ga(z',0,t)E(x-z')\,dz'=\int_\Si\Ga(x'-z',0,t)E(z',x_n)\,dz'\end{equation} and
\begin{equation}\label{eq_def_B}B(x,t)=\int_\Si\Ga(x-z',t)E(z',0)\,dz'=\int_\Si\Ga(z',x_n,t)E(x'-z',0)\,dz'.\end{equation}
They are defined in \cite{MR0171094} for $n=3$.
The estimates for $A, B$, and their derivatives are given in \cite[(62, 63)]{MR0171094} for $n=3$. For general dimension $n\ge2$, we can use the same approach and derive the following estimates for $l+n\ge3$:
\begin{equation}\label{eq_estA}|\pd_x^l\pd_t^mA(x,t)|\lesssim\frac1{t^{m+\frac12}(x^2+t)^{\frac{l+n-2}2}}\end{equation}
 and
\begin{equation}\label{eq_estB}|\pd_{x'}^l\pd_{x_n}^k\pd_t^mB(x,t)|\lesssim\frac1{(x^2+t)^{\frac{l+n-2}2}(x_n^2+t)^{\frac{k+1}2+m}}.\end{equation}
In fact, the last line of \cite[page 39]{MR0171094} gives
\begin{equation}\label{eq_estB2}|\pd_{x'}^l\pd_{x_n}^k B(x,t)|\lesssim\frac1{(x^2+t)^{\frac{l+n-2}2}t^{\frac{k+1}2}}\,e^{-\frac{x_n^2}{10t}}.\end{equation}

\begin{remark}\label{remark_est_AB}
For $n=2$, the condition $l\ge1$ is needed as $A(x,t)$ and $B(x,t)$ grow logarithmically as $|x|\to \infty$.
In fact, one may prove for $n=2$
\[
|A(x,t)|+|B(x,t)|\lesssim\frac{1 + |\log(|x_2|+\sqrt{t})|+|\log(|x_1|+|x_2|+\sqrt{t})|}{\sqrt{t}}.
\]
\end{remark}

Define as in \cite[(4.4)]{Green}
\EQ{\label{eq_def_Ci}
C_i(x,y,t)=\int_0^{x_n}\!\int_\Si\pd_n\Ga(x-y^*-z,t)\,\pd_iE(z)\,dz.
}
It is related to the Golovkin tensor via \eqref{Kijeq1} below, and to the Green tensor, see \cite[(4.6), (4.8)]{Green}.
We will prove in Appendix B (\S\ref{Sec6}) the following two lemmas for $C_i$.
\begin{lem}\label{lem2.5}
We have
\begin{align}
\label{eq0108c}
\pd_{y_n}C_i(x,y,t) & = \pd_{x_n}C_i(x,y,t)-\left(\pd_{y_n}e^{-\frac{y_n^2}{4t}}\right)\pd_iA(x-y',t), \quad\, \forall\, i\\
\label{eq0108a}
\pd_{x_n}C_i(x,y,t) &= \pd_{x_i}C_{n} (x,y,t) + \pd_{x_i}\pd_{x_n}B(x-y^*,t), \quad (i \not =n),
\\
\label{eq0108b}
\pd_{x_n}C_n(x,y,t) &= -\sum_{k<n}\pd_{x_k}C_k(x,y,t) - \frac12\, \pd_n\Ga(x-y^*,t).
\end{align}
\end{lem}

The special case of \eqref{eq0108a} and \eqref{eq0108b} for $y=0$ and $n=3$ are Solonnikov \cite[(67), (68)]{MR0171094}.

\begin{lem}\thlabel{Ci_estimate}
Let $n \ge 2$. For $x\in\R^n_+$, $y\in\R^n$ and $t>0$ we have
\begin{equation}\label{eq_Ci_estimate_yn}
\left|\pd_{x',y'}^l\pd_{x_n}^k\pd_{y_n}^{q}\pd_t^mC_i(x,y,t)\right|
\lesssim\frac{\exp\bke{-\frac{ ((y_n)_+)^2}{20t}}}{t^{m+\frac{q+1}2}\left(|x-y^*|^2+t\right)^{\frac{l+n-1}2}\left((x_n+y_n)^2+t\right)^{\frac{k}2}}.
\end{equation}
In particular, when $y=0$, we have
\begin{equation}\label{Ci-est}
\left|\pd_{x',y'}^l\pd_{x_n}^k\pd_t^mC_i(x,0,t)\right|
\lesssim\frac{1}{t^{m+\frac12} \left(|x|^2+t\right)^{\frac{l+n-1}2}\left(x_n^2+t\right)^{\frac{k}2}}.
\end{equation}
\end{lem}
This lemma is stated in \cite[Remark 5.2]{Green}. The special case $n=3$ and $y=0$ is \cite[(69)]{MR0171094}.

We may rewrite the definition \eqref{eq_def_Kij} of the Golovkin tensor $K_{ij}$ as
\EQ{
K_{ij}(x,t) = \wt K_{ij}(x,t)    -2\,\de_{nj}\pd_iE(x)\de(t),
}
where $\wt K_{ij}$ denotes the first two terms (and the function part) of \eqref{eq_def_Kij},
\EQ{
\label{Kijeq1}
\wt K_{ij}(x,t) = - 2\de_{ij}\pd_n \Ga(x,t) - 4 \pd_j C_i(x,0,t).
}
We have $K_{ij}(x,t) = \wt K_{ij}(x,t)$ for $t>0$. By Lemma \ref{lem2.5} with $y=0$,
we have the alternative formula when $j=n$,
\EQ{\label{Kijeq2}
\wt K_{in}(x,t) =
\begin{cases}
{-4  \pd_iC_n(x,0,t)  -4  \pd_i \pd_n B(x,t)  }&\quad i<n,\\[2mm]
{\sum_{k<n}4\pd_{k}C_k(x,0,t)}&\quad i=n.
\end{cases}
}

We can recover $K_{ij}$ estimates \eqref{eq_estKij} using \eqref{Kijeq1} for $j<n$, \eqref{Kijeq2} for $j=n$, and
$C_i$ estimate \eqref{Ci-est}. The key is that there is no $x_n$ derivative acting on $C_k$ in  \eqref{Kijeq1} for $j<n$ and \eqref{Kijeq2} for $j=n$.

Note that $A$ and $B$ defined in Solonnikov \cite[(60)-(61) on p.37]{MR0171094} and
 $C_i^{\text{Slnk}}(x,t)$ on \cite[(66) on p.40]{MR0171094} are only defined for $n=3$,  and differ from ours by a factor of $4\pi$: For $n=3$,
\EQ{
 A(x,t) = 4\pi A^{\text{Slnk}}(x,t) ,  \quad  B(x,t) = 4\pi B^{\text{Slnk}}(x,t), \quad  C_i(x,0,t) =  4\pi C_i^{\text{Slnk}}(x,t).
}

\section{Stokes flows}\label{Sec3}

In this section we give pointwise bounds of %
the solution {($\hat v(x,t),\hat p(x,t))$} of the Stokes system \eqref{E1.1} with boundary data $\phi$ defined on $\Si \times \R$ by
\eqref{boundary-data1}-\eqref{boundary-data3}.

Note that $\phi$ has only the normal component, and its tangential component is zero.
{Using \eqref{Golovkin} and \eqref{eq-hat-p-def},
$(\hat v, \hat p)$} is given by
\EQS{
\hat v_i(x,t)
&= \int_{-\infty}^\infty \int_{\Si} K_{in}(x-\xi',t-s) \phi_n(\xi',s)\, d\xi' ds\\
&=
 \int_0^t \int_{\Si} \wt K_{in}(x-\xi',t-s) \phi_n(\xi',s)\, d\xi' ds
 -2 \int_{\Si} \pd_iE(x-\xi') \phi_n(\xi',t)\, d\xi',
\label{hatv.formula}
}
{
and
\EQS{\label{hatp.formula}
\hat p(x,t)=&~2\pd_n^2\int_{\Si}E(x-\xi')\phi_n(\xi',t)\,d\xi'
+2\int_{\Si}E(x-\xi')\pd_t\phi_n(\xi',t)\,d\xi'\\
&-4(\pd_t-\De_{x'})\int_{-\infty}^\infty\int_{\Si}\pd_nA(x-\xi',t-\tau)\phi_n(\xi',\tau)\,d\xi'd\tau
}
}
where $\wt K_{in}(x,t) $ is given in \eqref{Kijeq1} and \eqref{Kijeq2}. As explained after \eqref{Kijeq2}, it has better estimates than \eqref{Kijeq1} when $j=n$.
Another advantage of \eqref{Kijeq2} over \eqref{Kijeq1} is that we can integrate by parts and move the tangential derivative to $\phi(\xi',s)$.

As $\phi(\xi',s)=0$ for $s \le 1/4$, we know $\hat v_i(x,t) =0$ for $t<1/4$ from these formulas.

\begin{prop}
\label{th3.1}
Let $n \ge 2$ and $\phi(\xi',s)$ be defined on $\Si \times \R$ by
\eqref{boundary-data1}--\eqref{boundary-data3} for $0<a\le 1/2$.
Let $\hat v(x,t)$ be the solution of the Stokes system \eqref{E1.1} with boundary data $\phi$, zero initial data and zero force, given by \eqref{hatv.formula}.
 Then
\EQ{\label{th31eq1}
| \hat v_i(x,t) | \le  \frac{CN_1}{\bka{x}^{n-1}} ,
}
\EQS{\label{th31eq2}
| \pd_{x_j} \hat v_i(x,t) | \le CN_2 \bigg[ &\frac{1}{\bka{x}^{n}}
 + \frac{\si \LN } {\bka{x} ^{n-1}(x_n+1)^{2a}}
 \\
& + \de_{n2} \one_{|x|<2}\, \si \LN \log\bke{2+ \frac{1}{x_2^2+|t-1|}} \bigg],
}
\EQ{\label{LN.def}
\LN =  (x_n^2+|t-1|)^{a-\frac12} + \de_{a=\frac12} \log\bke{2+ \frac{1}{x_n^2+|t-1|}} ,
}
for $x\in \R^n_+$ and $ t\in[0,2]$.
Above
\[
N_1=  \textstyle\sum_{j=0}^1 \|\nb_{\xi'}^j\phi\|_{L^\infty},\quad
N_2= \textstyle\sum_{j=0}^2 \|\nb_{\xi'}^j\phi\|_{L^\infty}+ \textstyle\sum_{j=0}^1 \norm{(1-s)^{1-a}\pd_s \nb_{\xi'}^j \phi(\xi',s)}_{L^\infty} ,
\]
$\si = \de_{i<n} \de_{jn}$ as in \eqref{eq_estKij},
and $C$ does not depend on $\phi$.
\end{prop}

\begin{remark}\label{rem3.2}
\EN{
\item[(i)] When $a<1/2$ and $x_n>2$, we have
$\LN \sim x_n^{2a-1}$, and \eqref{th31eq2} becomes
\EQ{
| \pd_{x_j} \hat v_i(x,t) | \le CN_2 \bigg[ \frac{1}{\bka{x}^{n}}
 + \frac{\si  } {\bka{x} ^{n-1}(x_n+1)} \bigg], \quad (x_n>2).
}
\item[(ii)]
For $n\ge2$ including $n=2$, our proof shows that
\EQ{\label{eq3.6}
\abs{\pd_{x_j} \hat v_i(x,t) - \si \pd_n I_2(x,t) } \le \frac{CN_2}{\bka{x}^n}
}
for $x\in \R^n_+$ and $ t\in[0,2]$.
Here $\pd_n I_2$ will be given in \eqref{dnI.dec} when $\si=1$, and is the main term in the estimate \eqref{th31eq2} when $\si = 1$.
}
\end{remark}

\begin{proof}\quad
{\bf Step 1}.\quad
We first estimate $\hat v_i(x,t)$ and we may assume $N_1\le1$. Denote formula \eqref{hatv.formula} for $\hat v_i$ as
$\hat v_i(x,t) = I+J$, where $I$ is the space-time integral, and $J$ is the space only integral.

We first consider
\[
J=  -2 \int_{\Si} \pd_iE(x-\xi') \phi_n(\xi',t)\, d\xi'.
\]
If $|x|\ge2$, we have
\[
|J|\lec \int_{B_1'} \frac1{|x-\xi'|^{n-1}} \, d\xi' \lec \frac1{|x|^{n-1}}.
\]
If $|x|<2$ and $i<n$, we can integrate by parts in $\xi_i$ to get
\[
|J|\lec   \int_{B_1'} |E(x-\xi')| \, d\xi' \lec  1
\]
for either $n\ge 3$ or $n=2$. If $|x|<2$ and $i=n$,
\[
|J| \lec \int_{B_1'} \frac{x_n}{|x-\xi'|^n}\, d\xi' \lec 1
\]
by  Lemma \ref{lemma2.1}. We conclude for $n\ge 2$,
\EQ{\label{0115a}
|J|\le \frac C{\bka{x}^{n-1}}, \quad (x\in \R^n_+,\ t\in\R).
}

We next consider
\[
I=\int_0^t \int_{\Si} \wt K_{in}(x-\xi',s)\, \phi_n(\xi',t-s)\, d\xi' ds.
\]
If we estimate it directly using \eqref{eq_estKij}, then for $|x|>2$ and $i<n$ (so that $\si=1$ in \eqref{eq_estKij}), we get
\EQS{\label{I-x>2}
I &\lec \int_0^t \int_{B_1'} \frac1{(|x-\xi'|+\sqrt{s})^{n-1} (x_n+\sqrt{s}) \sqrt{s}}\, d\xi'ds
\\
&\lec\frac{1}{|x|^{n-1}}  \int_0^t \frac1{(x_n+\sqrt{s})\sqrt{s}}\, ds
=\frac{2}{|x|^{n-1}}  \log(1+ \frac {\sqrt t}{x_n}) .
}
The last equality is by the change of variables $s=u^2$.
The log singularity in \eqref{I-x>2} is undesirable.

Alternatively,
we will use the formulas  for $\wt K_{in}$ in \eqref{Kijeq2}. First consider the case $i<n$.
We use \eqref{Kijeq2}$_1$ and integrate by parts in $\pd_{x_i}=-\pd_{\xi_i}$ the term $\pd_i C_n$ to get for $i<n$,
\EQS{
I &= -4
\int_0^t \int_{\Si}  C_{n}(x-\xi',0,s)\, \pd_i \phi_n(\xi',t-s)\, d\xi' ds
\\
&\qquad
-4 \int_0^t \int_{\Si}  \pd_i \pd_n B(x-\xi',s)\, \phi_n(\xi',t-s)\, d\xi' ds=: I_1 + I_2.
}

Using estimate \eqref{Ci-est} for $C_i$,
\EQ{\label{I1.est1}
|I_1| \lec \int_0^t \int_{B_1'} \frac1{(|x-\xi'|+\sqrt{s})^{n-1} \sqrt{s}}\, d\xi'ds .
}
If $|x|>2$, then $|x-\xi'| \sim |x|$ for $\xi'\in B_1'$ and %
\EQ{\label{I1.est2}
|I_1|\lec \int_0^t \frac1{|x|^{n-1} \sqrt{s}}\, ds
\lec \frac{\sqrt t}{|x|^{n-1}}.
}

If $|x|<2$, by \eqref{I1.est1}  and Lemma \ref{lemma2.1}, for $n \ge 3$,
\[
|I_1|\lec \int_{B_1'} \frac1{|x-\xi'|^{n-2}}\, d\xi' \lec 1,
\]
and for $n=2$,
\[
|I_1|\lec \int_{B_1'} 1 + \log_+  \frac t{|x-\xi'|^{2}}\, d\xi' \lec C_t.
\]
Together with \eqref{I1.est2}, we get
\EQ{\label{I1.est}
|I_1|\le \frac {C}{\bka{x}^{n-1}}, \quad (x\in \R^n_+,\ 0 <t\le 2).
}

For $I_2$ involving $\pd_nB$, we first substitute $\pd_nB(x,t) = -\frac{x_n}{2t} B(x,t)$ which follows from the definition \eqref{eq_def_B} of $B$. By estimate \eqref{eq_estB2} for $B$,
\EQ{\label{I2.est1}
|I_2 |
\lec \int_0^t \int_{B_1'} \frac{x_n}{s}\, \frac{e^{-\frac{x_n^2}{10s}}}{(|x-\xi'|+\sqrt{s})^{n-1}\sqrt{s}}\, d\xi'ds.
}

If $|x|>2$,
\Eq{
|I_2|
\lec   \frac1{|x|^{n-1}} \int_0^t \frac{x_n}s\, \frac{e^{-\frac{x_n^2}{10s}}}{\sqrt{s}}\, ds
=   \frac1{|x|^{n-1}}  \int_0^{t/x_n^2} \frac{e^{-\frac{1}{10\tau}}}{\tau^{3/2}}\, d\tau
\lec  \frac1{|x|^{n-1}}.
}
The replacement $\pd_nB= -\frac{x_n}{2t} B$ introduces a factor $\frac {x_n}{\sqrt s}$ in \eqref{I2.est1}. Without this factor, we get $I_2
\lec   \frac1{|x|^{n-1}} \int_0^{t/x_n^2} \frac{e^{-\frac{1}{10u}}}{u}\, du$, which is unbounded as $x_n \to 0_+$.

Suppose now $|x|<2$.
Noting that $B'_1\subset B'_3(x')$,
\EQN{
I_2
&=-4 \int_0^t \int_{B'_3(x')}  \pd_i \pd_n B(x-\xi',s)\, \phi_n(\xi',t-s)\, d\xi' ds
\\
&=-4 \int_0^t \int_{B'_3(x')}  \pd_i \pd_n B(x-\xi',s)\, [\phi_n(\xi',t-s)-\phi_n(x', t-s)]\, d\xi' ds,
}
where we have used $\int_{B'_3(x')}  \pd_i \pd_n B(x-\xi',s)d\xi'=0$.
By the mean value theorem,
\EQN{
|I_2 |
&\lec \int_0^t \int_{B'_3(x')} \abs{ \pd_i \pd_n B(x-\xi',s)}\,\norm{\nabla \phi_n (\cdot, t-s)}_{L^{\infty}_x}\abs{\xi'-x'} d\xi' ds
\\
&\lec \int_0^t \int_{B_3'(x')} \frac{x_n}{s}\, \frac{e^{-\frac{x_n^2}{10s}}}{(|x-\xi'|+\sqrt{s})^{n-1}\sqrt{s}}\abs{\xi'-x'}\, d\xi'ds
\\
&\lec \int_0^t \int_{B_3'} \frac{x_n}{s^{3/2}}\, \frac{e^{-\frac{x_n^2}{10s}}}{|\xi'|^{n-2}}\, d\xi'ds
\lec \int_0^t  \frac{x_n}{s^{3/2}}\, e^{-\frac{x_n^2}{10s}}\, ds\lec 1.
}

We conclude for $n\ge 2$, for all $i< n$,
\EQ{\label{0115b}
|I_2|\lec \frac 1{\bka{x}^{n-1}}  , \quad (x\in \R^n_+,\ t\in[0,2]).
}

When $i=n$, we rewrite $I$ using  \eqref{Kijeq2}$_2$ (instead of  \eqref{Kijeq2}$_1$)
and integration by parts. We get terms like $I_1$ but we do not get $I_2$. Thus, $I$ has the same estimate as \eqref{I1.est}.

The combination of \eqref{0115a}, %
\eqref{I1.est} and \eqref{0115b} gives \eqref{th31eq1}.

\bigskip

{\bf Step 2}.\quad
We now estimate $\pd_j \hat v_i(x,t)$ and we may assume $N_2\le1$. Using the same notation above, we have $\pd_j \hat v_i(x,t) = \pd_jI + \pd_jJ$, where $I$ is the space-time integral, and $J$ is the space only integral. %

We first consider
\[
\pd_jJ = -2 \int_\Si \pd_i\pd_jE(x-\xi') \phi_n(\xi',t)\, d\xi'.
\]
If $|x|\ge2$, we have
\[
|\pd_jJ| \lec \int_{B_1'} \frac1{|x-\xi'|^n}\, d\xi' \lec \frac1{|x|^n}.
\]
If $|x|<2$ and $i<n$, $j<n$, we can integrate by parts in $\xi_i$ and $\xi_j$ to get
\[
|\pd_jJ| \lec \int_{B_1'} |E(x-\xi')|\, d\xi' \lec 1
\]
for either $n\ge3$ or $n=2$. If $|x|<2$ and %
$j<n=i$, then we integrate by parts in $\xi_j$ to get
\[
|\pd_jJ| \lec \int_{B_1'} \frac{x_n}{|x-\xi'|^n}\, d\xi' \lec 1
\]
by  Lemma \ref{lemma2.1}. The case $i<n=j$ is the same. The case $i=j=n$ follows from $\div J=0$.
We conclude for $n\ge 2$,
\EQ{\label{0120a}
|\pd_jJ|\le \frac C{\bka{x}^{n}}, \quad (x\in \R^n_+,\ t\in\R).
}

We now consider
\[
\pd_jI = \int_0^t \int_\Si \pd_j \widetilde K_{in}(x-\xi',s) \phi_n(\xi',t-s)\, d\xi'ds
\]
and first consider $|x|\ge2$. If $i<n$, $j<n$, we use \eqref{Kijeq2}$_1$ and integrate by parts in $\pd_{x_i}=-\pd_{\xi_i}$ the term $\pd_{x_i}\pd_{x_j} C_n$ to get
\EQS{\label{def-pdjI}
\pd_jI &= -4
\int_0^t \int_{\Si}  \pd_{x_j}C_{n}(x-\xi',0,s)\, \pd_i \phi_n(\xi',t-s)\, d\xi' ds
\\
&\quad
-4 \int_0^t \int_{\Si}  \pd_i \pd_j \pd_n B(x-\xi',s)\, \phi_n(\xi',t-s)\, d\xi' ds
= \pd_jI_1 + \pd_jI_2.
}
Using estimate \eqref{Ci-est} for $\pd_{x_j}C_n$,
\EQ{\label{djI1.est1}
|\pd_jI_1| \lec \int_0^t \int_{B_1'} \frac1{(|x-\xi'|+\sqrt{s})^{n} \sqrt{s}}\, d\xi'ds \lec \int_0^t  \frac1{|x|^{n} \sqrt{s}}\, ds
\lec \frac{\sqrt{t}}{|x|^n}.
}

For $\pd_jI_2$ involving $\pd_nB$, we first substitute $\pd_nB(x,t) = -\frac{x_n}{2t} B(x,t)$ which follows from the definition \eqref{eq_def_B} of $B$. By estimate \eqref{eq_estB2} for $\pd_i\pd_jB$,
\[
|\pd_jI_2 |
\lec \int_0^t \int_{B_1'} \frac{x_n}{s}\, \frac{e^{-\frac{x_n^2}{10s}}}{(|x-\xi'|+\sqrt{s})^{n}\sqrt{s}}\, d\xi'ds.
\]
Thus,
\EQ{\label{djI2.est1}
|\pd_j I_2|
\lec \frac1{|x|^{n}} \int_0^t \frac{x_n}s\, \frac{e^{-\frac{x_n^2}{10s}}}{\sqrt{s}}\, ds
=   \frac1{|x|^{n}}  \int_0^{t/x_n^2} \frac{e^{-\frac{1}{10\tau}}}{\tau^{3/2}}\, d\tau
\lec  \frac1{|x|^{n}}.
}
It follows from \eqref{djI1.est1} and \eqref{djI2.est1} that for $i<n$ and $j<n$,
\EQ{\label{est-pdjI-<n<n}
|\pd_j I| \lec \frac1{|x|^{n}}.
}

If $|x|\ge2$ and $i=n$, $j<n$, we use \eqref{Kijeq2}$_2$ and integrate by parts in $\pd_{x_k}=-\pd_{\xi_k}$ the term $\pd_{x_j} \pd_{x_k} C_k$ to get
\EQ{\label{def-pdjI=n<n}
\pd_jI = 4
\int_0^t \int_{\Si}  \sum_{k<n} \pd_jC_{k}(x-\xi',0,s)\, \pd_k \phi_n(\xi',t-s)\, d\xi' ds.
}
Using estimate \eqref{Ci-est} for $\pd_{x_j}C_k$,
\EQ{\label{djI1.est2}
|\pd_jI| \lec \int_0^t \int_{B_1'} \frac1{(|x-\xi'|+\sqrt{s})^{n} \sqrt{s}}\, d\xi'ds \lec \frac{\sqrt{t}}{|x|^n}
}
for $i=n$, $j<n$.

If $|x|\ge2$ and $i<n$, $j=n$, we use \eqref{Kijeq2}$_1$ and integrate by parts in $\pd_{x_i}=-\pd_{\xi_i}$ the term $\pd_{x_i}\pd_{x_n} C_n$ to get
\EQS{\label{dnI.dec}
\pd_nI &= -4
\int_0^t \int_{\Si}  \pd_{x_n}C_{n}(x-\xi',0,s)\, \pd_i \phi_n(\xi',t-s)\, d\xi' ds
\\
& \quad -4 \int_0^t \int_{\Si}  \pd_i \pd_n^2 B(x-\xi',s)\, \phi_n(\xi',t-s)\, d\xi' ds= \pd_nI_1 + \pd_nI_2.
}
Using \eqref{eq0108b},
\EQ{\label{eq-def-dnI1}
\pd_nI_1 = 4 \int_0^t \int_{B_1'} \bkt{\sum_{k<n} \pd_{x_k}C_k(x-\xi',0,s) + \frac12\, \pd_n\Ga(x-\xi',s)} \pd_i\phi_n(\xi',t-s).
}
Using estimate \eqref{Ci-est} for $\pd_{x_j}C_k$,
\EQ{\label{est-pdnI1}
|\pd_nI_1|
\lec \int_0^t \int_{B_1'} \frac1{\sqrt{s}(|x-\xi'|+\sqrt{s})^n}+ \frac1{(|x-\xi'|+\sqrt{s})^{n+1}}\, d\xi'ds
\lec \frac{\sqrt t}{|x|^n}.
}
For $\pd_nI_2$, using $\pd_tB = \De B$, integration by parts, and $B(x,0)=0=\phi_n(\xi',0)$,
\EQS{\label{eq-def-dnI2}
\pd_nI_2
&= -4 \int_0^t \int_{B_1'} \pd_i \bke{-\sum_{k<n}\pd_k^2B(x-\xi',s) + \pd_sB(x-\xi',s)} \phi_n(\xi',t-s)\, d\xi'ds  \\
&= 4\int_0^t \int_{B_1'} \sum_{k<n} \pd_k^2B(x-\xi',s) \pd_i\phi_n(\xi',t-s)\, d\xi'ds\\
&\qquad - 4\int_0^t \int_{B_1'} \pd_iB(x-\xi',s) \pd_t\phi_n(\xi',t-s)\, d\xi'ds
=:I_{21} + I_{22} .
}
Using the estimate \eqref{eq_estB2},
\EQ{\label{est-I21}
|I_{21}| \lec  \int_0^t \int_{B_1'} \frac{e^{-\frac{x_n^2}{10s}}}{(|x-\xi'|+\sqrt{s})^n\sqrt{s}}\, d\xi'ds \lec  \frac{\sqrt t}{|x|^n}.
}

For $I_{22}$, we consider two cases: $t<1$ and $t>1$. If $t<1$, using the estimate \eqref{eq_estB2} and the definition of $\phi$ given in \eqref{boundary-data1} so that $|\pd_s \phi_n(\xi',s)| \le c g(\xi') (1-s)^{a-1}$
for $s\in (0,1)$,
\EQN{
|I_{22} |
&\lec \int_0^t \int_{B_1'} \frac{e^{-\frac{x_n^2}{10s}}}{(|x-\xi'|+\sqrt{s})^{n-1}\sqrt{s}} (1-t+s)^{a-1} \, d\xi'ds\\
&\lec \frac1{|x|^{n-1}} \int_0^t \frac{e^{-\frac{x_n^2}{10s}}}{\sqrt{s}(1-t+s)^{1-a}}\,ds
\lec \frac1{|x|^{n-1}} \int_0^t \frac{1}{\sqrt{s}(x_n^2+1-t+s)^{1-a}}\,ds,
}
where we have used $e^{-\frac{x_n^2}{10s}} \lec \min \big( 1, \big(\frac s{x_n^2+s}\big)^{1-a}\big)$ and
\[
\min  \bke{ 1, \frac s{x_n^2+s}}\cdot \frac1{1-t+s}
\le \min  \bke{ \frac1{1-t+s}, \frac 1{x_n^2+s}}\approx \frac 1{x_n^2+1-t+s}.
\]
By Lemma \ref{lemma2.1} with $L=1$, $d=1/2$, $a=x_n^2+|1-t|$ and $k=1-a$, we get
\EQN{
|I_{22} |
&\lec  \frac1{|x|^{n-1}}  \frac1{x_n+1} \bkt{\frac 1{(x_n^2+|1-t|)^{1/2-a}}
 +\de_{a=\frac12} \log \bke{2+ \frac 1{x_n^2+|1-t|}}}.
}

If $t>1$, using that $\phi(\xi',t-s)$ is supported in $t-1<s<t$ and $e^{-\frac{x_n^2}{10s}} \lec  \big(\frac s{x_n^2+s}\big)^{1/2}$,
\EQN{
|I_{22} |
&\lec \int_{t-1}^t \int_{B_1'} \frac{e^{-\frac{x_n^2}{10s}}}{(|x-\xi'|+\sqrt{s})^{n-1}\sqrt{s}} (1-t+s)^{a-1} \, d\xi'ds\\
& %
\lec \frac1{|x|^{n-1}} \int_{t-1}^t \frac{1}{\sqrt{x_n^2+s}(1-t+s)^{1-a}}\,ds
= \frac1{|x|^{n-1}} \int_0^1 \frac{1}{(x_n^2+t-1+ u)^{1/2}u^{1-a}}\,du.
}
By Lemma \ref{lemma2.1} with $L=1$, $d=a$, $a=x_n^2+|1-t|$ and $k=1/2$, we get
\EQN{
|I_{22} |
&\lec  \frac1{|x|^{n-1}}  \frac1{(x_n+1)^{2a}} \bkt{\frac 1{(x_n^2+|1-t|)^{1/2-a}}
 +\de_{a=\frac12} \log \bke{2+ \frac 1{x_n^2+|1-t|}}}.
}

Summarizing, for $0<a\le \frac 12$ and $t\in(0,2)$,
\EQ{\label{eq-CL0115-a}
|I_{22}| \lec \frac{\LN} {|x|^{n-1}(x_n+1)^{2a}} ,
}
where $\LN$ is as defined in \eqref{LN.def},
\[
\LN =  (x_n^2+|t-1|)^{a-\frac12} + \de_{a=\frac12} \log\bke{2+ \frac{1}{x_n^2+|t-1|}} .
\]

It follows from \eqref{est-pdnI1}, \eqref{est-I21} and \eqref{eq-CL0115-a} that for $i<n=j$ and $|x|\ge2$,
\EQ{\label{est_pdnI-<n=n}
|\pd_n I| \lec \frac1{|x|^n} + \frac{\LN}{|x|^{n-1}(x_n+1)^{2a}} .
}

For $i=j=n$, we use $\div I=0$. Hence $\pd_n I$ has the same estimate as \eqref{est-pdjI-<n<n}.

Combining \eqref{0120a}, \eqref{est-pdjI-<n<n}, \eqref{djI1.est2}, and \eqref{est_pdnI-<n=n}, we get that for $|x|\ge2$ and all $i,j$,
\EQ{\label{eq-est-nav->2}
|\pd_j \hat v_i | \lec \frac1{|x|^{n}} + \frac{\si \LN} {|x|^{n-1}(x_n+1)^{2a}} ,
}
where $\si = \de_{i<n} \de_{jn}$ as defined in \eqref{eq_estKij}. We also get \eqref{eq3.6} when $|x|\ge 2$.

\bigskip

Suppose now $|x|<2$. If $i<n$ and $j<n$, integrating by parts in $\pd_{x_j}=-\pd_{\xi_j}$ both terms $\pd_jI_1$ and $\pd_jI_2$ in \eqref{def-pdjI}, we obtain, using the same estimates for $\hat v$ itself leading to \eqref{I1.est} and \eqref{0115b},
\EQS{\label{eq-est-pdjI-<n<n}
|\pd_j I| &\lec
\abs{\int_0^t \int_{\Si}  C_{n}(x-\xi',0,s)\, \pd_i \pd_{j}\phi_n(\xi',t-s)\, d\xi' ds}
\\
&\quad +\abs{ \int_0^t \int_{\Si}  \pd_i  \pd_n B(x-\xi',s)\, \pd_{j} \phi_n(\xi',t-s)\, d\xi' ds} \lec 1.
}

If $i=j=n$, using $\div I=0$, $\pd_n I$ as the same estimate as \eqref{eq-est-pdjI-<n<n}.

If $i=n$ and $j<n$, integrating by parts in $\pd_{x_j}=-\pd_{\xi_j}$ the term $\pd_{x_j}C_k$ in \eqref{def-pdjI=n<n}, we get
\EQ{\label{est-pdjI-=n<n}
|\pd_jI| \lec   \int_0^t \int_{B_1'} \sum_{k<n}|C_k(x-\xi',0,s)|\, d\xi'ds \lec  1.
}

If $i<n$ and $j=n$, we recall $\pd_nI = \pd_nI_1 + \pd_n I_2$, where $\pd_nI_1$ and $\pd_nI_2$ are given in \eqref{eq-def-dnI1} and \eqref{eq-def-dnI2}, respectively. Integrating by parts in $\pd_{x_k}=-\pd_{\xi_k}$ the term $\pd_{x_k}C_k$ in \eqref{eq-def-dnI1}, we get
\EQN{
|\pd_nI_1|
&\lec   \int_0^t \int_{B_1'} |C_k(x-\xi',0,s)|\, d\xi'ds +   \int_0^t \int_{B_1'} |\pd_n\Ga(x-\xi',s)|\, d\xi'ds = I_{11}+I_{12}.
}
We have $I_{11} \lec 1$ as before. Using change of variables $z=\frac{\xi'}{\sqrt s}$ and $u=s/x_n^2$,
\EQN{
I_{12}
&\lec       \int_0^t \int_{B_1'} \frac{x_n}s\, \frac{e^{-\frac{|x-\xi'|^2}{4s}}}{s^{\frac{n}2}}\, d\xi'ds
\lec       \int_0^\infty \int_{\R^{n-1}} \frac{x_n}se^{-\frac{x_n^2}{4s}} \, \frac{e^{-\frac{|\xi'|^2}{4s}}}{s^{\frac{n}2}}\, d\xi'ds \\
&= \int_0^t\frac{x_n}{s^{3/2}}e^{-\frac{x_n^2}{4s}}\,ds\,\int_{\R^{n-1}}e^{-z^2/4}dz
= \int_0^\infty\frac{C}{u^{3/2}}e^{-\frac{1}{4u}} du = C.
}

For $\pd_nI_2$, we recall that $\pd_nI_2=I_{21} + I_{22}$, where $I_{21}$ and $I_{22}$ are defined in \eqref{eq-def-dnI2}. Integrating by parts in $\pd_{x_k}=-\pd_{\xi_k}$ (exactly once) the term $\pd_k^2B$ in $I_{21}$, we derive
\EQN{
|I_{21} | &\lec   \int_0^t \int_{B_1'}\sum_{k<n}|\pd_kB(x-\xi',s)|\, d\xi'ds
\lec   \int_0^t \int_{B_1'} \frac{e^{-\frac{x_n^2}{10s}}}{(|x-\xi'|+\sqrt{s})^{n-1}\sqrt{s}}\, d\xi'ds \\
&\lec   \int_0^t \log\bke{2+\frac1{x_n+\sqrt{s}}} \frac1{\sqrt{s}}\, ds
\lec   \int_0^{\sqrt 2} \log\bke{2+\frac1u} du \lec  1.
}

For $I_{22}$, by the same computation as for \eqref{eq-CL0115-a}, we obtain%
\EQN{
|I_{22} |
&\lec \int_{(t-1)_+} ^t \int_{B_1'} \frac{e^{-\frac{x_n^2}{10s}}}{(|x-\xi'|+\sqrt{s})^{n-1}\sqrt{s}} (1-t+s)^{a-1} \, d\xi'ds\\
&\lec \int_{(t-1)_+} ^t \int_{B_3'} \frac{1}{(|\xi'|+x_n+\sqrt{s})^{n-1}}d\xi'\, s^{-1/2} e^{-\frac{x_n^2}{10s}} (1-t+s)^{a-1} \, ds\\
&\lec \int_{(t-1)_+} ^t \log \bke{2 + \frac 1{x_n+\sqrt{s}}}\, s^{-1/2} e^{-\frac{x_n^2}{20s}} (1-t+s)^{a-1} \, ds.
}
If we bound
$
\log \bke{2 + \frac 1{x_n+\sqrt{s}}} \le \log \bke{2 + \frac 1{x_n}}
$
then the same estimate for \eqref{eq-CL0115-a} goes through and gives
\EQ{\label{0224b}
|I_{22} |\lec \log \bke{2 + \frac 1{x_n}} \LN.
}
We now try to improve the above estimate.

By integration by parts,
\EQ{\label{0224c}
| I_{22} |
\lec \int_0^t \int_{B_1'} |B(x-\xi',s)| |\pd_i\pd_t\phi_n(\xi',t-s)| \,d\xi'ds .
}
For $n\ge3$, noting $\norm{(1-s)^{1-a}\nb_{\xi'}\pd_s \phi_n(\xi',s)}_\infty\lec N_2\le 1 $,
\EQS{\label{I22-local-n3}
| I_{22} |
&\lec \int_{(t-1)_+}^t \int_{B_1'} \frac{e^{-\frac{x_n^2}{10s}}}{\bke{|x-\xi'|+\sqrt{s}}^{n-2} \sqrt{s}}\, (1-t+s)^{a-1}\, d\xi'ds\\
&\lec \int_{(t-1)_+}^t e^{-\frac{x_n^2}{10s}}\frac{1}{ \sqrt{s}}\, (1-t+s)^{a-1}\, ds
\quad
\lec \LN,
}
by the same estimates leading to \eqref{eq-CL0115-a}.

For $n=2$, if $x_2^2>|t-1|$, then by \eqref{0224b},
\[
| I_{22} | \lec \log\bke{2+\frac1{x_2}} \LN \lec \log\bke{2+\frac1{x_2^2+|t-1|}} \LN.
\]
We now consider $n=2$ and $x_2^2<|t-1|$.
Using \eqref{0224c} and Remark \ref{remark_est_AB},
\EQN{
| I_{22} |
&\lec \int_{(t-1)_+}^t \int_{B_3'} \frac{1+|\log(x_2+\sqrt{s})|+|\log(|\xi_1|+x_2+\sqrt{s})|}{\sqrt{s}}\, (1-t+s)^{a-1}\, d\xi_1ds.
}
Note that
\[
1+ |\log y| \lec \log \frac{20}y, \quad 0<y<10.
\]
Thus
\EQ{\label{0224d}
| I_{22} |
\lec \int_{(t-1)_+}^t\frac{\log \frac{20}{x_2+\sqrt{s}}}{\sqrt{s}}\, (1-t+s)^{a-1}\, ds.
}

If $0<t<1$, by changing variables $s=4u^2$ and using \eqref{0223a} of Lemma \ref{th2.3} with $m=2-2a$,
\EQN{
| I_{22} |
&\lec \int_0^t \frac{\log \frac{4}s}{\sqrt{s}(1-t+s)^{1-a}}\, ds
\lec \int_0^{1/2} \frac{-\log u}{(\sqrt{1-t}+u)^{2-2a}}\, du
\\
&\approx \int_0^{1/2} \frac{-\log u}{(x_2+\sqrt{1-t}+u)^{2-2a}}\, du
\\
&
\lec 1 + (x_2+\sqrt{1-t})^{2a-1} + \de_{a=\frac12} \bkt{ \log (x_2+\sqrt{1-t})}^{2}.
}

If $1<t<2$, then $\log \frac{20}{x_2+\sqrt{s}}\le A=\log \frac{20}{x_2+\sqrt{t-1}}$ in \eqref{0224d} and
\EQN{
| I_{22} |
&\lec \int_{t-1}^t \frac{A}{\sqrt{s}(1-t+s)^{1-a}}\, ds
 = \int_0^1 \frac{A}{\sqrt{u+t-1} u^{1-a}}\, du
 \approx \int_0^1 \frac{A}{\sqrt{u+x_n^2+t-1} u^{1-a}}\, du.
}
By Lemma \ref{lemma2.1} with $L=1$, $d=a$, $a=x_n^2+|1-t|$ and $k=1/2$, we get
\EQN{
|I_{22} |
&\lec \log \frac{20}{x_2+\sqrt{t-1}} \bkt{\frac 1{(x_n^2+|1-t|)^{1/2-a}}
 +\de_{a=\frac12} \log \bke{2+ \frac 1{x_n^2+|1-t|}}}.
}

Combining all cases, we have shown for $n=2$ and $i<2=j$,
\EQ{\label{eq-est-pdnI1}
|I_{22}|\lec \log\bke{2+\frac1{x_2^2+|t-1|}} \LN, \quad \forall |x|<2, \quad \forall |t-1|<1.
}
and the same estimate for $\pd_nI_1$.

Therefore, \eqref{0120a}, \eqref{eq-est-pdjI-<n<n}, \eqref{est-pdjI-=n<n}, and \eqref{eq-est-pdnI1} imply that for $|x|<2$
\EQ{\label{eq-est-nav-<2}
|\pd_j \hat v_i| \lec 1 + \si \LN \bkt{1+  \de_{n2} \log\bke{2+ \frac{1}{x_2^2+|t-1|}} },
}
where $\si = \de_{i<n} \de_{jn}$.
The gradient estimate \eqref{th31eq2} then follows from the outer estimate \eqref{eq-est-nav->2} and the inner estimate \eqref{eq-est-nav-<2}. We also get \eqref{eq3.6} when $|x|< 2$.
\end{proof}

\medskip

We will need the following lemma for the proof of Theorem \ref{th1.1}. It extends Lemma 1 and Lemma 2 of \cite{Kang2005}, which treat the case $n=3$, to general dimensions.

\begin{lem}\label{th3.2}
Let $n \ge 3$ and \
$\displaystyle
K(x', t)=\int_{\Si} \frac{e^{-\frac{|x'-z'|^2}{4t}}}{\abs{z'}^{n-2}} dz'$\quad for $x'\in\Si= \R^{n-1}$.
For any $m\in [2,\infty)$,
we have
\begin{align}\label{estimate-K1}
K(x',s)&\ge
\bke{\frac{m}{(m+1){d}}}^{n-2}s^{\frac{n-1}{2}} (4\pi)^{\frac{n-1}2} \bke{1- 2^{\frac{n-1}4} e^{-\frac{{d}^2}{8m^2s}}},
\\
\label{estimate-K2}
K(x',s)&\le
\bke{\frac{m}{(m-1){d}}}^{n-2}s^{\frac{n-1}{2}} (4\pi)^{\frac{n-1}2} + \frac{C}{d^{n-2}} s^{\frac{n-1}{2}} e^{-\frac{{d}^2}{8m^2s}}.
\end{align}
where ${d}=\abs{x'}>0$, and $C=C(n)>0$ is some absolute constant.
\end{lem}

\emph{Remark.}\quad For the application to the proof of Theorem \ref{th1.1}, it is crucial  that the two leading constants $(\frac{m}{m+1})^{n-2} (4\pi)^{\frac{n-1}2}$ and $(\frac{m}{m-1})^{n-2} (4\pi)^{\frac{n-1}2}$ are close for large $m$. In particular, we cannot absorb the second terms in \eqref{estimate-K1} and \eqref{estimate-K2} to the first terms.

\begin{proof}
Due to the scaling property $K(\mu x', \mu^2 s)=\mu K(x', s)$ and that of \eqref{estimate-K1} and \eqref{estimate-K2}, it suffices to compute for the case $s=1$.
We split $\Si$ into three disjoint regions $\mathcal A$, $\mathcal B$ and $\mathcal C$ defined by
\[
\mathcal A=\bket{z'\in\Si: \abs{x'-z'}<\frac{{d}}{m}},
\quad
\mathcal B=\bket{z'\in\Si: \abs{z'}<\frac{(m-1)d}{m}},
\]
and
$ \mathcal C=\Si\setminus (\mathcal A \cup \mathcal B)$,
and divide the integral as
\[
K(x',1)=\int_{\Si} \frac{e^{-\frac{|x'-z'|^2}{4}}}{\abs{z'}^{n-2}} dz'=\int_{\mathcal A}\cdots+\int_{\mathcal B}\cdots+\int_{\mathcal C}\cdots:=K_{\mathcal A}+K_{\mathcal B}+K_{\mathcal C}.
\]

Since $\abs{z'}<(1-1/m){d}$ in $\mathcal B$, we have $|x'-z'|\ge \frac dm$ and, with $m \ge 2$,
\EQN{
K_{\mathcal B} &= \bke{ \int_{|z'|< \frac d2}+  \int_{ \frac d2<|z'|<d-\frac dm}} \frac{e^{-\frac{|x'-z'|^2}{4}}}{\abs{z'}^{n-2}} dz' \\
&
\le e^{-\frac{{d}^2}{16}}  \int_{|z'|< \frac d2} \frac{dz'}{\abs{z'}^{n-2}}
+ (d/2)^{2-n} \int _{ \frac d2<|z'|<d-\frac dm} e^{-\frac{{d}^2}{8m^2}} e^{-\frac{|x'-z'|^2}{8}} dz'
\\
&\le C d e^{-\frac{{d}^2}{16}}  +C d^{2-n}e^{-\frac{{d}^2}{8m^2}}
\\
&\le C d^{2-n}e^{-\frac{{d}^2}{8m^2}}.
}

Since $(1-1/m){d}\le \abs{z'}$ in both $\mathcal A$ and   $\mathcal C$, we have
\EQN{
K_{\mathcal A}+K_{\mathcal C} &\le \bke{\frac{m}{(m-1){d}}}^{n-2}\int_{\mathcal A\cup \mathcal C}
e^{-\frac{|x'-z'|^2}{4}}dz'
\le \bke{\frac{m}{(m-1){d}}}^{n-2}\int_{\Si} e^{-\frac{|w'|^2}{4}}dw',%
}
and
\[
 \int_\Si e^{-|w'|^2/4} dw' =
\int_\Si (4\pi)^{\frac{n-1}2} \Ga'(w',1) \,dw' = (4\pi )^{\frac{n-1}2} .
\]
Above $\Ga'(w',t)$ is the heat kernel on $\Si \times (0,\infty)$.
Summing up the above estimates, we obtain the upper bound of $K$ in \eqref{estimate-K2}.

For the lower bound \eqref{estimate-K1}, we note that $ \abs{z'}\le (1+1/m){d}$ in $\mathcal A$. Thus
\[
K(x',1) \ge K_{\mathcal A}\ge \bke{\frac{m}{(m+1){d}}}^{n-2}\int_{|w'|<\frac dm} e^{-\frac{|w'|^2}{4}}dw',
\]
while
\EQN{
\int_{|w'|<\frac{{d}}{m}} e^{-\frac{|w'|^2}{4}}dw'
&= \int_{\Si} e^{-\frac{|w'|^2}{4}}dw' - \int_{|w'|>\frac{{d}}{m}} e^{-\frac{|w'|^2}{4}}dw'
\\
&\ge   \int_{\Si} e^{-\frac{|w'|^2}{4}}dw' -  e^{-\frac{{d}^2}{8m^2}} \int_{\Si} e^{-\frac{|w'|^2}{8}}dw'
\\
&=  (4\pi)^{\frac{n-1}2} - e^{-\frac{{d}^2}{8m^2}}  (4\pi \sqrt 2)^{\frac{n-1}2},
}
using
\[
 \int_\Si e^{-|w'|^2/8} dw' =
\int_\Si (4\pi \sqrt 2)^{\frac{n-1}2} \Ga'(w',\sqrt 2) \,dw' = (4\pi \sqrt 2)^{\frac{n-1}2} .
\]
This shows  \eqref{estimate-K1}.
\end{proof}

We are now ready to prove Theorem \ref{th1.1}.
\begin{proof}[Proof of Theorem \ref{th1.1}]

\mbox{}\medskip

{\bf Step 1}.\quad
H\"older continuity.

Note that the boundary value $\phi \in C^{2a,a}(\R^{n-1}\times [0,2])$. Let $b=a$ if $0<a<1/2$, and $b$ can be any number in $(0,1/2)$ when $a=1/2$. We claim that, since $\phi(\xi',s) = g(\xi') \,h(s)\, e_n$ and $g\in C^1_c(B_1')$,
\EQ{\label{0507}
R' \phi  \in C^{2b,b}(\R^{n-1}\times [0,2]),
}
where $R'=(R'_1,\ldots,R'_{n-1})$ is $(n-1)$-dimensional Riesz transforms on $\R^{n-1}$. (If $n=2$, $R'$ is the Hilbert transform on $\R$.)
Indeed, it is known that homogeneous H\"older continuity is preserved by the Riesz transforms, i.e. $R' g \in \dot{C}^{2b}(\R^{n-1})$ (see e.g.~Chang-Jin \cite[Proposition 4.1]{MR3398794} and related reference therein) if $0<2b<1$.
Being locally H\"older continuous, $R'g$ is locally bounded.
If $\abs{\xi'}>2$, then
\[
\abs{R_i g(\xi')}=\abs{\int_{\R^{n-1}}c\frac{\xi'_i-y'_i}{|\xi'-y'|^{n}} g(y')dy'}\lec
\int_{B_1'}\frac{1}{|\xi'|^{n-1}} \abs{g(y')}dy'\lesssim\frac{\norm{g}_{L^{1}}}{|\xi'|^{n-1}} .
\]
The above shows \eqref{0507}.
By \cite[Theorem 1.2]{MR3398794}, we have $\hat v\in C^{2b,b}(\overline{\R^n_+}\times [0,2])$.

\medskip

{\bf Step 2}.\quad
Finite global energy in $\R^n_+\times (0,2)$.\smallskip

We first show that $\hat v(t) \in L^2_x (\R^n_+)$, $n\ge 3$ for $t\in [0,2]$.
With the aid of the estimate \eqref{th31eq1} in Proposition \ref{th3.1}, we have
\[
\int_{\R^n_+} |\hat v(x,t)|^2\,dx \le C\int_{\R^n_+} \frac{N^2_1}{\bka{x}^{2(n-1)}} \, dx \lesssim \int_0^{\infty}\frac{1}{(\rho^2+1)^{\frac{n-1}{2}}}d\rho<\infty
\]
for $n\ge 3$. The last integral diverges if $n=2$.
On the other hand, by the estimate \eqref{th31eq2} of the gradient of $\hat v$ for $n\ge 3$,
\begin{equation}\label{grad-v-thm1}
| \pd_{x_j} \hat v_i(x,t) | \le CN_2 \bigg[ \frac{1}{\bka{x}^{n}}
 + \frac{\si \LN } {\bka{x} ^{n-1}(x^2_n+1)^{a}} \bigg], \qquad \si = \de_{i<n} \de_{jn},
\end{equation}
where
\[
\LN =  (x_n^2+|t-1|)^{a-\frac12} + \de_{a=\frac12} \log\bke{2+ \frac{1}{x_n^2+|t-1|}} .
\]
It is clear that $1/{\bka{x}^{n}}$ is square integrable in $\R^n_+$, uniformly in $t$, and thus it suffices to estimate the second term of the right-hand side in \eqref{grad-v-thm1},  in the case $\sigma=1$.
Firstly, we consider the case that $a<1/2$. Using $ x_n^2+|t-1|\le x_n^2+1$, we obtain
\EQS{
&\int_0^2\int_{\R^n_+} \frac{(x_n^2+|t-1|)^{2a-1}} {\bka{x} ^{2(n-1)}(x^2_n+1)^{2a}}\,dx dt
\lesssim
\int_0^2\int_{\R_+} \frac{1} {(x_n^2+|t-1|)}\, dx_n dt
\\
&
\lesssim
\int_0^2\frac{1} {|t-1|^{\frac{1}{2}}}\, dt<\infty.
}

In the case $a=1/2$, we have
\EQ{
\int_0^2\int_{\R^n_+} \frac{\log\bke{2+ \frac{1}{x_n^2+|t-1|}}^2} {\bka{x} ^{2(n-1)}(x^2_n+1)}\,dx dt
\lesssim \int_0^2\int_{\R_+}\frac{\log\bke{2+ \frac{1}{|t-1|}}^2} {(x^2_n+1)}\, dx_n dt<\infty.
}

Collecting the above estimates, we get \eqref{th1.1eq1}, i.e., the solution is of finite energy when $n \ge 3$.

\medskip

{\bf Step 3}.\quad Pressure estimate.

Let $n\ge2$. We now derive the pointwise estimate \eqref{est-p-hat} for the pressure $\hat p$. From \eqref{hatp.formula}, we have
\EQN{
\hat p(x,t)=&~2\pd_n^2\int_{\Si}E(x-\xi')\phi_n(\xi',t)\,d\xi'
+2\int_{\Si}E(x-\xi')\pd_t\phi_n(\xi',t)\,d\xi'\\
&-4(\pd_t-\De_{x'})\int_{-\infty}^\infty\int_{\Si}\pd_nA(x-\xi',t-\tau)\phi_n(\xi',\tau)\,d\xi'd\tau\\&
=:I + II + III.
}
Since
\[
I = 2\sum_{i=1}^{n-1} \int_{\Si}E(x-\xi')\pd_i^2\phi(\xi',t)\, d\xi',
\]
\[
| I | \lec h(t) \sum_{i=1}^n \int_{B_1'}|E(x-\xi')|\, d\xi'.
\]
For $|x|<2$,
\[
\int_{B_1'}|E(x-\xi')|\, d\xi' \lec 1.
\]
For $|x|\ge2$, then $|x-\xi'|\ge |x|-|\xi'|\ge1$ so that $\log|x-\xi'|>0$ for $|\xi'|\le1$, and
\EQN{
\int_{B_1'}|E(x-\xi')|\, d\xi' \lec
\begin{cases}
\int_{B_1'}\frac1{|x-\xi'|^{n-2}}\, d\xi' \lec \frac1{|x|^{n-2}}&\quad n\ge3,\\
\int_{B_1'} \log|x-\xi'|\, d\xi' \lec \log(|x|+1)&\quad n=2.
\end{cases}
}
Thus,
\[
| I | \lec h(t) \bke{\one_{n\ge3}\, \frac1{\bka{x}^{n-2}} + \one_{n=2} \log(|x|+1)}.
\]
For $II$, from the same computation as for $I$ we have
\EQS{\label{est-p-hat-II}
| II | &\lec |h'(t)| \bke{\one_{n\ge3}\, \frac1{\bka{x}^{n-2}} + \one_{n=2} (1+\log(|x|+1))}\\
&\lec \frac{\chi_{\frac14\le t\le1}(t) }{(1-t)^{1-a}} \bke{\one_{n\ge3}\, \frac1{\bka{x}^{n-2}} + \one_{n=2} \log(|x|+2)}.
}
Note that the estimate of $I$ is dominated by that of $II$.

For $III$, we consider the three cases: $t<\frac14$, $t\in[\frac14,1)$ and $t\ge1$.

If $t<\frac14$, since $\supp h \subset [\frac14,1]$ and $A(x,t)=0$ for $t<0$, we have $t-\tau<0$ for $t<\frac14$ and $\tau\ge\frac14$. So
\EQN{
\int_{-\infty}^\infty\int_{\Si}\pd_nA(x-\xi',t-\tau)\phi_n(\xi',\tau)\,d\xi'd\tau
&=\int_{-\infty}^\infty\int_{\Si}\pd_nA(x-\xi',t-\tau)g(\xi')h(\tau)\,d\xi'd\tau \\
&= \int_{\frac14}^\infty\int_{\Si}\pd_nA(x-\xi',t-\tau)g(\xi')h(\tau)\,d\xi'd\tau
=0,
}
and thus $III=0$.

If $t\in[\frac14,1)$, integrating by parts we have
\EQN{
| III | &\lec \int_{-\infty}^\infty \int_{\Si} |\pd_nA(x-\xi',t-\tau)| \bkt{|\De_{\xi'} g(\xi')| h(\tau) + |g(\xi')| |h'(\tau)|} d\xi'd\tau\\
& \lec \int_{\frac14}^t \int_{B_1'} \frac1{\sqrt{t-\tau}(|x-\xi'|+\sqrt{t-\tau})^{n-1}}\, h(\tau)\, d\xi'd\tau\\
&\quad + \int_{\frac14}^t \int_{B_1'} \frac1{\sqrt{t-\tau}(|x-\xi'|+\sqrt{t-\tau})^{n-1}}\, |h'(\tau)|\, d\xi'd\tau
=: III_1 + III_2.
}
Since
\EQN{
&\int_{B_1'} \frac1{(|x-\xi'|+\sqrt{t-\tau})^{n-1}}\, d\xi'
\\
&\lec
\begin{cases}
\int_0^{|x|+1} \frac{r^{n-2}}{(t+\sqrt{t-\tau})^{n-1}}\, d\xi' \lec \frac{|x|+1}{|x|+\sqrt{t-\tau}+1} \bke{1+\log_+\frac{|x|+1}{\sqrt{t-\tau}}}&\qquad |x|<2,\\
\frac1{(|x|+\sqrt{t-\tau})^{n-1}}&\qquad |x|\ge2,
\end{cases}\\
&\lec \frac{\log\bke{2+\frac1{\sqrt{t-\tau}}}}{(|x|+\sqrt{t-\tau}+1)^{n-1}},
}
we have, using the change of variables $u=\sqrt{t-\tau}$,
\EQN{
| III_1| &\lec \int_{\frac14}^t \frac1{\sqrt{t-\tau}}\, \frac{\log\bke{2+\frac1{\sqrt{t-\tau}}}}{(|x|+\sqrt{t-\tau}+1)^{n-1}}\, (1-\tau)^a\, d\tau\\
& \lec \int_0^{\sqrt{t-\frac14}} \frac{\log\bke{2+\frac1u}}{(|x|+u+1)^{n-1}}\, du
\lec \frac1{\bka{x}^{n-1}},
}
and
\EQN{
| III_2 | &\lec \int_{\frac14}^t \frac1{\sqrt{t-\tau}}\, \frac{\log\bke{2+\frac1{\sqrt{t-\tau}}}}{(|x|+\sqrt{t-\tau}+1)^{n-1}}\, \frac1{(1-\tau)^{1-a}}\, d\tau\\
&\lec \frac1{\bka{x}^{n-1}} \int_0^{\sqrt{t-\frac14}} \log\bke{2+\frac1u}\, \frac1{(u^2+1-t)^{1-a}}\, du
\lec \frac1{(1-t)^{1-a}\bka{x}^{n-1}}.
}
Therefore, for $t\in[\frac14,1)$
\EQ{\label{est-CL0615-a}
| III | \lec \frac1{(1-t)^{1-a}\bka{x}^{n-1}}.
}

If $t\ge1$, integrating by parts we have
\EQN{
| III | &\lec \int_{-\infty}^\infty \int_{\Si} |\pd_nA(x-\xi',t-\tau)| \bkt{|\De_{\xi'} g(\xi')| h(\tau) + |g(\xi')| |h'(\tau)|} d\xi'd\tau\\
& \lec \int_{\frac14}^1 \int_{B_1'} \frac1{\sqrt{t-\tau}(|x-\xi'|+\sqrt{t-\tau})^{n-1}}\, h(\tau)\, d\xi'd\tau\\
&\quad + \int_{\frac14}^1 \int_{B_1'} \frac1{\sqrt{t-\tau}(|x-\xi'|+\sqrt{t-\tau})^{n-1}}\, |h'(\tau)|\, d\xi'd\tau
=: III_1 + III_2,
}
where, using the change of variables $u=\sqrt{t-\tau}$,
\EQN{
| III_1 | &\lec \int_{\frac14}^1 \frac1{\sqrt{t-\tau}}\, \frac{\log\bke{2+\frac1{\sqrt{t-\tau}}}}{(|x|+\sqrt{t-\tau}+1)^{n-1}}\, (1-\tau)^a\, d\tau\\
& \lec \int_{\sqrt{t-1}}^{\sqrt{t-\frac14}} \frac{\log\bke{2+\frac1u}}{(|x|+u+1)^{n-1}}\, du
\lec \frac1{\bka{x}^{n-1}},
}
and
\EQN{
| III_2 | &\lec \int_{\frac14}^1 \frac1{\sqrt{t-\tau}}\, \frac{\log\bke{2+\frac1{\sqrt{t-\tau}}}}{(|x|+\sqrt{t-\tau}+1)^{n-1}}\, \frac1{(1-\tau)^{1-a}}\, d\tau\\
&\lec \frac1{\bka{x}^{n-1}} \int_{t-1}^{t-\frac14} \frac1{\sqrt{s}}\, \log\bke{2+\frac1{\sqrt{s}}} \frac1{(1-t+s)^{1-a}}\, ds, \qquad s=t-\tau,\\
&\lec \frac1{\bka{x}^{n-1}} \int_{t-1}^{t-\frac14} \frac1{s}\, \frac1{(1-t+s)^{1-a}}\, ds\\
&= \frac1{\bka{x}^{n-1}} \int_0^{\frac34} \frac1{(u+t-1)u^{1-a}}\, du,\qquad u=1-t+s,\\
&\lec \frac1{\bka{x}^{n-1}}\, \frac{(3/4)^a}{(t-1/4)^a}\, \frac1{(t-1)^{1-a}}
\sim \frac1{(t-1)^{1-a}\bka{x}^{n-1}}.
}
Therefore, for $t\ge 1$
\EQ{\label{est-CL0615-b}
| III | \lec \frac1{(t-1)^{1-a}\bka{x}^{n-1}}.
}
Combining \eqref{est-CL0615-a} and \eqref{est-CL0615-b}, for all $t$
\EQ{\label{est-p-hat-III}
| III | \lec \frac{\chi_{t\ge\frac14}(t)}{|t-1|^{1-a}\bka{x}^{n-1}}.
}
The pointwise estimate \eqref{est-p-hat} follows from \eqref{est-p-hat-II} and \eqref{est-p-hat-III}.

\medskip

{\bf Step 4}.\quad
Lower bound of velocity gradient: Set up.\smallskip

By \eqref{eq3.6} of Remark \ref{rem3.2}, we have $|\pd_j \hat v_i(x,1) - \si \pd_n I_2(x,1)| \le C \bka{x}^{-n}$ where $\si = \de_{i<n} \de_{jn}$.  In particular, $|\pd_j \hat v_i(x,1)|$ may blow up as $x_n \to 0_+$ only if $\si=1$, i.e., $i<n=j$.
We will show that $|\pd_n \hat v_i(x,1)|$ is indeed unbounded when $i<n$.
It suffices to estimate $\pd_n I_2 (x,1)$ given in \eqref{dnI.dec}. After integration by parts,
\EQ{\label{eq3.43}
\pd_n I_2 (x,1)= -4 \int_{0}^1 \int_{\Si}   \pd_n^2 B(x-\xi',s)\, \pd_i\phi_n(\xi',1-s)\, d\xi' ds, \quad
\text{for }i<n.
}
Recall $K(x', t)$ defined in Lemma \ref{th3.2},
\[
K(x', t):=\int_{\Si} \frac{e^{-\frac{|x'-z'|^2}{4t}}}{\abs{z'}^{n-2}} dz',\qquad x'\in\R^{n-1}.
\]
By the definition \eqref{eq_def_B} of $B$, it follows that
\EQN{
4B(x-\xi',s)&=e^{-\frac{x_n^2}{4s}}
\frac{C_n}{s^{\frac{n}{2}}}\,K(x'-\xi',s),
\\
4\pd_n^2 B(x-\xi',s)&=e^{-\frac{x_n^2}{4s}}\bke{-\frac12+\frac{x_n^2}{4 s}}
\frac{C_n}{s^{\frac{n+2}{2}}}\,K(x'-\xi',s),
}
where $C_n=4(4\pi)^{-n/2} [n(n-2)|B_1|]^{-1}>0$.
By our assumption \eqref{boundary-data1}--\eqref{boundary-data3} on the boundary value, we have
\EQ{\label{0306a}
\pd_i\phi_n(\xi',1-s)= \pd_i g(\xi')\, s^a, \quad (0\le a\le\tfrac12).
}
Therefore we have
\EQ{\label{eq3.45}
\pd_n I_2 (x,1)=C_n\int_{0}^1 \int_{\Si}   e^{-\frac{x_n^2}{4s}}\bke{\frac{1}{2}-\frac{x_n^2}{4 s}}
\frac{1}{s^{\frac{n+2}{2}}}\, K(x'-\xi',s)\, \pd_ig(\xi') s^a \, d\xi' ds.
}

To get a lower bound of $\pd_n I_2$, we choose $g$ in a product form as in \eqref{gxiform}.

For fixed $i$,
denote the reflection $ x'_* =  (x_1, \ldots , -x_i, \ldots ,x_{n-1})$ which differs from $x'$ in the sign of the $i$th component $x_i$. For our choice of $g$ we have
\EQ{\label{eq3.47}
\pd_n I_2( x_*',x_n,1) =- \pd_n I_2(x,1) .
}
That is, $\pd_n I_2(x,1)$ is odd in $x_i$ for the fixed $i$ in \eqref{eq3.45}.
It is because
\EQN{
&  \int_\Si
 K(x'-\xi',s)\, \pd_ig(\xi')\, d\xi'
=  \int_\Si
 K(x'_*-\xi'_*,s)\, \pd_ig(\xi')\, d\xi'\\
&=  \int_\Si
 K(x'_*-\eta',s)\, \pd_ig(\eta'_*)\, d\eta'
=-\int_\Si
 K( x'_*-\eta',s)\, \pd_ig(\eta')\, d\eta' .
}
Above $\eta' = \xi'_*$, and we used $\pd_ig(\eta'_*) = -\pd_i g(\eta')$ in the last equality.

In what follows we assume
\EQ{\label{xi.cond}
|x_i| = \max_{1\le j \le n-1} |x_j| \quad \text{and}\quad x_i<0.
}
The case $x_i>0$ follows from \eqref{eq3.47} with a sign change.

\medskip

{\bf Step 5}.\quad Key estimates.

Decompose $\pd_n I_2$ in \eqref{eq3.45} as
\EQN{
\frac {\pd_n I_2(x,1)}{C_n} &=
 \int_{0}^{x^2_n/2}\int_{\Si}   e^{-\frac{x_n^2}{4s}}s^{a-1 -\frac n2}
   K(x'-\xi',s)\bke{-\frac{1}{2}+\frac{x_n^2}{4 s}}
\bkt{ (\pd_i g)_- - (\pd_i g)_+} (\xi') \, d\xi' ds
\\
& \quad + \int_{x^2_n/2}^1 \int_{\Si}  e^{-\frac{x_n^2}{4s}}
s^{a-1 -\frac n2}\,   K(x'-\xi',s)\bke{\frac{1}{2}-\frac{x_n^2}{4 s}}
\bkt{   (\pd_ig)_+  -(\pd_ig)_-  }(\xi') \, d\xi' ds.
}
By Lemma \ref{th3.2} with a large enough $m$ that will be chosen later,
\EQN{
\frac {\pd_n I_2(x,1)}{C_n(4\pi)^{\frac{n-1}2} }
 &\ge \int_{0}^{x^2_n/2}\int_{\Si}   e^{-\frac{x_n^2}{4s}}s^{a - 1- \frac n2}
  \bke{-\frac{1}{2}+\frac{x_n^2}{4 s}} s^{\frac{n-1}2} \\
&\qquad\qquad\quad \cdot \bkt{\bke{\frac{m}{(m+1){d}}}^{n-2} (\pd_ig)_- - \bke{\frac{m}{(m-1){d}}}^{n-2} (\pd_ig)_+} (\xi') \, d\xi' ds
\\
&\quad  + \int_{x^2_n/2}^1 \int_{\Si}  e^{-\frac{x_n^2}{4s}} s^{a-1-\frac n2}\,  \bke{\frac{1}{2}-\frac{x_n^2}{4 s}} s^{\frac{n-1}2}\\
&\qquad\qquad \cdot \bkt{  \bke{\frac{m}{(m+1){d}}}^{n-2} (\pd_ig)_+  - \bke{\frac{m}{(m-1){d}}}^{n-2} (\pd_ig)_-  }(\xi') \, d\xi' ds
- J_{3},
}
where $d=|x' -\xi'|$ and
\EQN{
J_3 &= C \int_{0}^1\int_{\Si}   e^{-\frac{x_n^2}{4s}}\, \frac1{s^{\frac32 - a}}
  \abs{-\frac{1}{2}+\frac{x_n^2}{4 s}} \frac1{d^{n-2}}\, e^{-\frac{{d}^2}{8m^2s}} |\pd_i g|(\xi') \, d\xi' ds.
}
Thus, with $d_*=|x'_* -\xi'|$,
\EQS{\label{eq-0531-a}
&\frac {\pd_n I_2(x,1)}{C_n(4\pi)^{\frac{n-1}2} } + J_{3}
\\
&\ge  K_2 \int_{0}^{x^2_n/2}  e^{-\frac{x_n^2}{4s}}s^{a -\frac 32}
  \bke{-\frac{1}{2}+\frac{x_n^2}{4 s}}  ds
 + K_1  \int_{x^2_n/2}^1  e^{-\frac{x_n^2}{4s}} s^{a-\frac 32}\,  \bke{\frac{1}{2}-\frac{x_n^2}{4 s}}
 ds
\\
& = \bke{\frac{x_n^2}{4}}^{a-\frac 12} \Bigg[ K_2 \int_{1/2}^{\infty}  e^{-\sigma} \bke{ \sigma ^{\frac 12-a} - \frac12 \sigma ^{ -\frac 12-a}}d  \sigma
+ K_1  \int_{x_n^2/4}^{1/2}  e^{-\sigma}  \bke{\frac12 \sigma ^{ -\frac 12-a} - \sigma ^{\frac 12-a}  } d  \sigma \Bigg]
}
where $\si= \frac{x_n^2}{4s}$,
\EQ{\label{K1-def}
K_1 =
 \int_{B'_1, \xi_i>0}
\bkt{\bke{\frac{m}{(m+1){d_*}}}^{n-2} -\bke{\frac{m}{(m-1){d}}}^{n-2}}|\pd_i g(\xi')| \, d\xi',
}
and
\EQN{
K_2&=\int_{\Si ,\xi_i>0}
\bkt{\bke{\frac{m}{(m+1){d}}}^{n-2}  - \bke{\frac{m}{(m-1){d_*}}}^{n-2}} |\pd_ig (\xi')| \, d\xi' <0.
}
That $K_2<0$ is because $d>d_*$ for $x_i<0$ and $\xi_i>0$.
We now show a positive lower bound of $K_1$.
As $x_i<0$ by \eqref{xi.cond},  $d_*<d$ if $\xi_i>0$ and
\[
d-d_*
=
|x' -\xi'|-|x'_* -\xi'|
=\frac{|x'-\xi'|^2-|x' _* -\xi'|^2}{|x' -\xi'|+|x'_* -\xi'|}
=\frac{-4x_i\xi_i}{|x' -\xi'|+|x'_* -\xi'|}.
\]
On the support of $\pd_i g(\xi')$ with $\xi_i>0$, we have
\[
\xi_i\in\bke{\frac1{2\sqrt{n-1}},\frac4{5\sqrt{n-1}} },\quad |\xi'|\le 1.
\]
By \eqref{xi.cond}, $-x_i \ge \frac 1{\sqrt{n-1}}\,|x'|$. Also using $|x'|\ge 3$,
\EQ{\label{ineq-d_*-d}
d-d_*
\ge \frac{4\cdot\frac1{\sqrt{n-1}}\,|x'|\cdot\frac1{2\sqrt{n-1}}}{2\bke{|x'|+ 1}}
\ge \frac3{4 (n-1)}=:\delta.
}
Thus
\EQS{\label{eq3.54}
D &:=\bke{\frac{m}{(m+1){d_*}}}^{n-2} -\bke{\frac{m}{(m-1){d}}}^{n-2}
\\
& \ge \bke{\frac{m}{(m+1)(d-\de)}}^{n-2} -    \bke{\frac{m}{(m-1) d}}^{n-2}
\\
&= S \bke{\frac{m}{(m+1)(d-\de)}-    \frac{m}{(m-1) d}}
= \frac{ Sm \bke{(m+1)\delta-2d}}{(m+1)(d-\de)(m-1)d},
}
where
\EQS{\label{eq3.55}
S = \sum_{k=0}^{n-3} \bke{\frac{m}{(m+1)(d-\de)}}^{k}     \bke{\frac{m}{(m-1)d}}^{n-3-k}\approx \frac1{|x'|^{n-3}}
}
for $m>2$.
Choosing $m=4 (n-1) |x'| + 1$, we have
\EQ{
D \ge \frac{mS}{(m+1)(d-\de)(m-1)d} \bke{3|x'| - 2 (|x'| + 1) } \ge \frac {C_1}{|x'|^{n-1}},
}
for $|x'|\ge3$
 and $\xi' \in \supp \pd_i g$, $\xi_1>0$, with a constant $C_1$ independent of $|x'|\ge3$. Back to $K_1$ we have
\EQN{
K_1 &\ge \frac {C_1}{|x'|^{n-1}} \int_{B'_1, \xi_i>0}   |\pd_ig(\xi')|\, d\xi' = \frac {C_2}{|x'|^{n-1}}.
}

For $a<1/2$, we claim
\EQ{\label{M1M2def}
M_2:= \int_{1/2}^{\infty}  e^{-\sigma} \bke{ \sigma ^{\frac 12-a} - \frac12 \sigma ^{ -\frac 12-a}} d  \sigma
<  M_1:=\int_{0}^{1/2}  e^{-\sigma}  \bke{\frac12 \sigma ^{ -\frac 12-a} - \sigma ^{\frac 12-a}  } d  \sigma
}
which is equivalent to
\EQ{\label{eq3.60}
2\int_0^\infty  e^{-\sigma}\sigma^{\frac{1}{2}-a} d\sigma<  \int_0^\infty  e^{-\sigma}\sigma^{-\frac{1}{2}-a} d\sigma.
}
By integration by parts,
\EQN{
\text{LHS of \eqref{eq3.60}} &= \bkt{-2 e^{-\sigma}\sigma^{\frac{1}{2}-a}}_0^\infty  -   \int_0^\infty -2 e^{-\sigma}(\tfrac{1}{2}-a)\sigma^{-\frac{1}{2}-a} d\sigma
\\
&=0+  (1-2a)\text{RHS} < \text{RHS of \eqref{eq3.60}}.
}
Thus
\[
K_2 M_2 + K_1 M_1 >0
\]
if we take $m$ sufficiently large, $m \ge \frac {C(n)|x'|M_1}{M_1-M_2}$. In fact, to show $K_2 M_2 + K_1 M_1 >0$, we only need positivity of
\EQN{%
S=\bkt{\bke{\frac{m}{(m+1)d} }^{n-2} - \bke{ \frac{m}{(m-1)d_*} }^{n-2} }\th
+  \bkt{\bke{\frac{m}{(m+1)d_*} }^{n-2} - \bke{ \frac{m}{(m-1)d} }^{n-2} } >0,
}
where $\th= \frac{M_2}{M_1} \in (0,1)$. Let $A = d^{2-n}<(d_*+\de) ^{2-n} < B= d_*^{2-n} <1$.
Using
\[
1-C_1 \al \le \bke{\frac 1{1+\al}}^{n-2} \le 1-C_2 \al,\quad
1+C_3 \al \le \bke{\frac 1{1-\al}}^{n-2} \le 1+C_4 \al,
\]
for $\al \in (0,1/2)$,
where $C_i = C_i(n)$, and taking $\al = 1/m$,  we have
\EQN{
S &\ge (1-C_1\al)\th A - (1+C_4\al)\th B + (1-C_1 \al)B - (1+C_4\al) A
\\
&= (1-\th)(B-A) - \al (C_1 \th A + C_4 \th B + C_1 B + C_4 A)
\\
&> (1-\th)(B-A) - \al (C_1 + C_4)(A+ B ).
}
Thus $S>0$ if
\EQ{\label{lai-3.58}
\frac 1m = \al < \frac{(1-\th)(B-A)} {(C_1 + C_4)(B+A)},
}
uniformly in $\xi'$ in the support of $\pd_ig(\xi')$, $\xi_i>0$.
We have
\[
B-A
= \frac 1{d_*^{n-2}} - \frac 1{d^{n-2}}
\ge \frac 1{d_*^{n-2}} - \frac 1{(d_*+\de)^{n-2}}
\approx\frac 1{d_*^{n-3}} \bke{ \frac 1{d_*} - \frac 1{(d_*+\de)} }\approx\frac 1{d_*^{n-1}}
\approx \frac1{|x'|^{n-1}},
\]
and
$
B+A  \approx\frac 1{d^{n-2}} \approx \frac1{|x'|^{n-2}}.
$
Thus it suffices to take
\[
m = \frac {C(n) |x'|} {1-\th} = \frac {C(n)M_1}{M_1-M_2}\, |x'|.
\]
Then $S \gec B-A \gec |x'|^{1-n}$ and
\EQN{
K_1M_1 + K_2M_2
&= \int_{\supp \pd_i g, \xi_i>0} M_1 S |\pd_ig (\xi')|\, d\xi'\gec |x'|^{1-n} \int |\pd_ig (\xi')|\, d\xi',
}
and so
\EQS{\label{eq-0622-a}
\frac{\pd_n I_2(x,1)}{C_n(4\pi)^{\frac{n-1}2}} + J_3
&\gtrsim \bke{\frac{x_n^2}{4}}^{a-\frac 12} \Bigg[ K_2 M_2 + K_1 M_1 - K_1 \int_0^{{x_n^2/4}}  e^{-\sigma}  \bke{\frac12 \sigma ^{ -\frac 12-a} - \sigma ^{\frac 12-a}  } d  \sigma \Bigg]
\\
&\gtrsim \bke{\frac{x_n^2}{4}}^{a-\frac 12} \Bigg[|x'|^{1-n} - C |x'|^{2-n} x_n ^{1-2a}\Bigg]
\gtrsim |x'|^{1-n} x_n^{2a-1}- C |x'|^{2-n}
}
for $a<1/2$.

For $a=1/2$, using $K_1,\,|K_2| \lec |x'|^{2-n}$ and $K_1\gtrsim |x'|^{1-n}$, \eqref{eq-0531-a} becomes
\EQS{\label{eq-0622-b}
\frac {\pd_n I_2(x,1)}{C_n(4\pi)^{\frac{n-1}2} } + J_{3}
&\ge K_2 \int_{1/2}^{\infty}  e^{-\sigma} \bke{ 1 - \frac12 \sigma ^{ -1}}d  \sigma
+ K_1  \int_{{x_n^2/4}}^{1/2}  e^{-\sigma}  \bke{\frac12 \sigma ^{ -1} - 1  } d  \sigma \\
&\ge -C|x'|^{2-n} + \frac{K_1}2 \int_{x_n^2/4}^{1/2} e^{-\si} \si^{-1}\, d\si\\
&\ge -C|x'|^{2-n} + \frac{K_1}2\, e^{-\frac12} \int_{x_n^2/4}^{1/2} \si^{-1}\, d\si\\
&\ge -C|x'|^{2-n} + C|x'|^{1-n} \log \frac2{x_n}.
}

As for $J_3$, since $e^{-\frac{x_n^2}{4s}}\abs{-1/2 + x_n^2/(4s)}\lec 1$, we have
\[
|J_3|
\lec \int_0^1 \int_{\Si} \frac1{s^{\frac32 - a}}\, \frac1{d^{n-2}}\, e^{-\frac{d^2}{8m^2s}}\, |\pd_ig|(\xi')d\xi'ds.
\]
Using
\[
\int_0^1 \frac1{s^{\frac12 - a}} e^{-b/s}\frac {ds}s  = \int_b^\infty \bke{\frac vb}^{\frac12 - a} e^{-v} \frac {dv}v \le C b^{a-\frac12 },
\]
and
\[
\int_0^1 e^{-b/s}\frac {ds}s  = \int_b^\infty  e^{-v} \frac {dv}v \approx |\log b|
\]
for $b >0$, we get for $m \ge d$ and $b= \frac{d^2}{8m^2}$,
\EQS{\label{J3.est}
|J_3|
&\lec \int_{\supp \pd_i g}  \frac1{d^{n-2}}\,  \bkt{ \one_{a<\frac12}\, \bke{\frac{d^2}{8m^2}}^{a-\frac12 } + \one_{a=\frac12}\, \log \frac{8m^2}{d^2} } |\pd_i g|(\xi')\, d\xi'\\
&\lec \one_{a<\frac12}\, \frac{m^{1-2a} } {|x'|^{n-1 - 2a}} +  \one_{a=\frac12}\, \frac1{|x'|^{n-2}}\, \log\frac{4m}{|x'|}.
}
Choosing the same $m=C|x'|$, we get $|J_{3}|\lec \frac1{|x'|^{n-2}}$.
We conclude
\EQN{
\pd_n I_2(x,1)
\gtrsim |x'|^{1-n} \bke{ \one_{a<\frac12}\, x_n^{2a-1} + \one_{a=\frac12}\, \log \frac2{x_n}} - \frac C{|x'|^{n-2}}.
}
This completes the proof of \eqref{th1.1eq2}.
\end{proof}

\begin{proof}[Proof of Proposition \ref{th1.2}]\quad

We repeat Steps 1 and 2 of the proof of
Theorem \ref{th1.1}. In particular, $\pd_n I_2(x,1)$ is odd in $x_i$ as shown in \eqref{eq3.47}. Hence
$\pd_n I_2(x,1)=0$  if $x_i=0$, and it suffices to consider $x_i<0$. Recall $|x'|>3$, and we do not assume $|x_i| = \max_{1\le j \le n-1} |x_j|$.

We repeat Step 3 of the proof of
Theorem \ref{th1.1} up to \eqref{ineq-d_*-d}, and replace \eqref{ineq-d_*-d} by
\EQS{\label{delta.def}
d-d_*
&=\frac{-4x_i\xi_i}{|x' -\xi'|+|x'_* -\xi'|}
\\
&\ge \frac{4|x_i|\cdot\frac1{2\sqrt{n-1}}}{2\bke{|x'|+ 1}}
=\frac{|x_i|}{\sqrt{n-1}(1+|x'|)}=:\delta(|x_i|,|x'|),
}
for $\xi'$ in the support of $\pd_i g(\xi')$ with $\xi_i>0$. The estimate of $K_1$ defined in \eqref{K1-def} now proceeds as
\EQS{\label{K1-def-lai}
K_1 &=
 \int_{B'_1, \xi_i>0}
\bkt{\bke{\frac{m}{(m+1){d_*}}}^{n-2} -\bke{\frac{m}{(m-1){d}}}^{n-2}}|\pd_i g(\xi')| \, d\xi'\\&
\geq\int_{B'_1, \xi_i>0}
\bkt{\bke{\frac{m}{(m+1)(d-\delta)}}^{n-2} -\bke{\frac{m}{(m-1){d}}}^{n-2}}|\pd_i g(\xi')| \, d\xi' \\&
\ge \frac{C}{m|x'|^{n-1}}  \int_{B'_1, \xi_i>0} \bkt{(m+1)\de - 2d}
|\pd_i g(\xi')| \, d\xi'
}
using \eqref{eq3.54} and \eqref{eq3.55}. We now choose
\EQ{\label{m.choice2}
m = \frac 3\de(|x'|+1) = \frac{3\sqrt{n-1}(1+|x'|)^2}{|x_i|}.
}
Then
\[
K_1 \ge \frac{C}{m|x'|^{n-2}}  \int_{B'_1, \xi_i>0}
|\pd_i g(\xi')| \, d\xi'\ge \frac{C|x_i|}{|x'|^{n}}  \norm{\pd_i g}_{L^1}.
\]

For $a<1/2$, recall from the proof of
Theorem \ref{th1.1} we have
\[
\frac {\pd_n I_2(x,1)}{C_n(4\pi)^{\frac{n-1}2} } + J_{3} \ge  \bke{\frac{x_n^2}{4}}^{a-\frac 12}\bke{ K_2 M_2 + K_1 M_1  - C K_1 x_n ^{1-2a}}.
\]
To make $K_2 M_2 + K_1 M_1>0$ we take $m= C \frac{B+A}{B-A}$ by \eqref{lai-3.58}, while for
$A = d^{2-n}<(d_*+\de) ^{2-n} < B= d_*^{2-n} <1$ we have
\[
B-A
= \frac 1{d_*^{n-2}} - \frac 1{d^{n-2}}
\ge \frac 1{d_*^{n-2}} - \frac 1{(d_*+\de)^{n-2}}
\approx\frac 1{d_*^{n-3}} \bke{ \frac 1{d_*} - \frac 1{(d_*+\de)} }\approx\frac {\de}{d_*^{n-1}}.
\]
By \eqref{delta.def}, $\de = C|x_i|/|x'|$. Thus
\EQ{\label{m.choice3}
B-A \approx \frac {|x_i|}{|x'|^{n}}, \quad m = \frac {C|x'|^2}{|x_i|}\approx C \frac{B+A}{B-A} .
}
In this case,
 $S \gec B-A \gec \frac {|x_i|}{|x'|^{n}}$ and
\EQN{
K_1M_1 + K_2M_2
&= \int_{\supp \pd_i g, \xi_i>0} M_1 S |\pd_ig (\xi')|\, d\xi'\gec \frac {|x_i|}{|x'|^{n}} \int |\pd_ig (\xi')|\, d\xi'.
}
The error term
\[
  \bke{\frac{x_n^2}{4}}^{a-\frac 12}C K_1 x_n ^{1-2a} = C K_1 \lec |x'|^{2-n} .
\]

For $a=1/2$, using $K_1,\,|K_2| \lec |x'|^{2-n}$ and $K_1\gtrsim \frac{|x_i|}{|x'|^n}$, \eqref{eq-0531-a} becomes
\EQN{
\frac {\pd_n I_2(x,1)}{C_n(4\pi)^{\frac{n-1}2} } + J_{3}
&\ge K_2 \int_{1/2}^{\infty}  e^{-\sigma} \bke{ 1 - \frac12 \sigma ^{ -1}}d  \sigma
+ K_1  \int_{{x_n^2/4}}^{1/2}  e^{-\sigma}  \bke{\frac12 \sigma ^{ -1} - 1  } d  \sigma \\
&\ge -C|x'|^{2-n} + \frac{K_1}2 \int_{x_n^2/4}^{1/2} e^{-\si} \si^{-1}\, d\si\\
&\ge -C|x'|^{2-n} + \frac{K_1}2\, e^{-\frac12} \int_{x_n^2/4}^{1/2} \si^{-1}\, d\si\\
&\ge -C|x'|^{2-n} + C\,\frac{|x_i|}{|x'|^n}\, \log \frac2{x_n}.
}

As for $J_{3}$, by the estimate \eqref{J3.est} and the choice \eqref{m.choice3} of $m$, we have
\EQN{
|J_{3}|
&\lec \one_{a<\frac12}\, \frac{m^{1-2a} } {|x'|^{n-1 - 2a}} + \one_{a=\frac12}\, \frac1{|x'|^{n-2}}\, \log\frac{4m}{|x'|}\\
& \lec \one_{a<\frac12}\, \frac{C}{|x_i|^{1-2a} |x'|^{n-3+2a}} + \one_{a=\frac12}\, \frac1{|x'|^{n-2}}\, \log \frac{4|x'|}{|x_i|},
}
which is no smaller than $C|x'|^{2-n}$.
From \eqref{eq3.6} of Remark \ref{rem3.2}, the above discussion and the fact  $x_i>0$ follows from \eqref{eq3.47} with a sign change, we conclude for $x_i\neq0$
\EQN{
 |\pd_{x_n} \hat v_i(x,1)| &\geq \frac{C|x_i|}{|x'|^{n}}\, \bke{\one_{a<\frac12}\, x_n^{2a-1} + \one_{a=\frac12}\, \log\frac2{x_n} }\\
 &\quad - C  \bkt{\one_{a<\frac12}\, \frac{C}{|x_i|^{1-2a} |x'|^{n-3+2a}} + \one_{a=\frac12}\, \frac1{|x'|^{n-2}}\, \log \frac{4|x'|}{|x_i|} }.
}
This completes the proof of \eqref{th1.1eq3}.
\end{proof}

\begin{remark}
In the case $a=1/2$, H\"older continuity of $\hat v$ was proved up to boundary away from the region with non-zero boundary flux in Kang \cite{Kang2005}. The case $0<a<1/2$ can be proved similarly. However, it would require extra work to show H\"older continuity up to the support of the boundary flux, and it is easier to refer to Chang-Jin \cite{MR3398794}.
\end{remark}

\section{Navier-Stokes flows}
\label{Sec4}

Let $n \ge 3$ and $\phi(\xi',s)$ be defined on $\Si \times \R$ by
\eqref{boundary-data1}-\eqref{boundary-data3}.
Let $\hat v(x,t)$ be the Stokes flow with boundary data $\phi(x',t)$ and initial data $\hat v(x,0)=0$,
given by \eqref{hatv.formula}. Our aim in this section is to construct a solution $u(x,t)$ of the Navier-Stokes equations \eqref{NS} with similar properties of the form
\EQ{
    u = \al \hat v +v, \quad v= O(\al^2),
}
for $\al>0$ sufficiently small. Since $\hat v$ solves the Stokes system \eqref{E1.1}, $v$ solves
\EQ{\label{eq-stokes}
\pd_t v - \De v + \nb \pi = - (\al \hat v + v) \cdot \nb (\al \hat v + v), \quad \div v=0,
}
with zero initial and boundary values
\[
v(x,0)=0, \quad
v(x',0,t)=0.
\]

Since the main term of $u$ is $\al \hat v$, we expect $u$ to satisfy the same bounds of $\hat v$.
By Proposition \ref{th3.1}, $\hat v(x,t)$ satisfies the decay estimates
\EQ{\label{eq-0415-a}
| \hat v(x,t) | \lec \frac1{\bka{x}^{n-1}}\quad \text{ and }\quad
| \pd_{x_j} \hat v_i(x,t) | \lec \frac{1}{\bka{x}^{n}}
 + \frac{\si \LN(x,t) } {\bka{x} ^{n-1}(x_n+1)^{2a}},
}
where $\si = \de_{i<n} \de_{jn}$ and $\LN$ is defined in \eqref{LN.def}.
Hence we define $X$ and $Y$ to be the sets of functions defined in $\R^n_+ \times (0,2)$ with finite norms
\EQS{\label{XTYT.def}
\norm{f}_{{X}} &= \sup_{(x,t)\in\R^n_+\times (0,2)} \bke{ |f(x,t)| + |\na f(x,t)|} \bka{x}^n,\\
\norm{f}_{{Y}} &= \sup_{(x,t)\in\R^n_+\times (0,2)} \bkt{ |f(x,t)|  \bka{x}^{n-1} + |\na f(x,t)|\bke{ \frac{1}{\bka{x}^{n}}
 + \frac{ \LN (x,t)} {\bka{x} ^{n-1}(x_n+1)^{2a}}}^{-1}}.
}
We have $\hat v \in {Y}$ and $\norm{f}_{{Y}} \le \norm{f}_{{X}}$ (with constant 1).

Consider the bilinear map
\EQ{
B(f,g)_i(x,t) = \int_0^t \int_{\R^n_+} \pd_{y_k} G_{ij}(x,y,t-s) f_k(y,s) g_j (y,s)\, dy\,ds.
}

\begin{lem}\thlabel{lem4.1}
There is an absolute constant $C_1>0$ such that, if $f,g\in {Y}$, then
\[
\norm{ B(f,g) }_{{X}} \le C_1 \norm{f}_{{Y}} \norm{g}_{{Y}}.
\]
\end{lem}

\begin{proof}
We may assume $\norm{f}_{{Y}} \le 1$ and $\norm{g}_{{Y}}\le 1$.
By $\nb_{x,y} G_{ij}$ estimates \thref{thm3},
\EQS{\label{0426a}
|\nb_{x,y} G_{ij}(x,y,s) | &\lec \frac1{(|x-y|+\sqrt{s})^{n+1}} + \frac1{(|x^*-y|+\sqrt{s})^n (x_n+y_n+\sqrt{s})}
\\ &\lec  \frac1{(|x-y|+\sqrt{s})^n\sqrt{s}} .
}
Changing variables $s \to t-s$ and using Lemma \ref{lemma2.2} and $n \ge 3$,
\begin{align*}
|B(f,g)(x,t)|
&\lec\int_0^t \frac1{\sqrt{s}} \int_{\R^n_+} \frac1{(|x-y|+\sqrt{s})^n} \frac1{(|y|+1)^{2n-2}}\, dyds  \\
&\lec \int_0^t \frac1{\sqrt{s}} \bkt{  \frac1{(|x|+\sqrt{s}+1)^{2n-2}}\, \log \frac{|x|+\sqrt{s}+1}{\sqrt{s}} + \frac1{(|x|+\sqrt{s}+1)^{n}} } ds   \\
&\lec \int_0^{t}\bke{\frac{\log(|x|+3)+ |\log \sqrt s|}{(|x|+1)^{2n-2}}  + \frac{1}{(|x|+1)^n}}
\frac{ds}{\sqrt{s}}.
\end{align*}
Since the time integral is uniformly bounded in $t$,
\EQ{\label{0426b}
|B(f,g)(x,t)|  \lec \frac{\log (|x|+3)}{\bka{x}^{2n-2}} + \frac1{\bka{x}^n}\lec \frac1{\bka{x}^n}.
}

For $\na B(f,g)$, we use the formula
\[
\na B(f,g)(x,t) = - \int_0^t \int_{\R^n_+} \na_x G_{ij}(x,y,t-s) \pd_k [f_k(y,s)  g_j(y,s)] \, dy\,ds.
\]
Using  $\norm{f}_{{Y}} \le 1$ and $\norm{g}_{{Y}}\le 1$,  we have
\[
\pd_k [f_k(y,s)  g_j(y,s)] \lec \frac{1}{\bka{y}^{2n-1}}
 + \frac{ \LN(y,s) } {\bka{y} ^{2n-2}(y_n+1)^{2a}},
\]
where $\LN$ is defined in \eqref{LN.def}.
By  $\nb_{x,y} G_{ij}$ estimates \eqref{0426a},
\EQN{
&|\na B(f,g)(x,t)| \lec
 \int_0^t \int_{\R^n_+} \frac1{(|x-y|+\sqrt{t-s})^n \sqrt{t-s}} \frac1{(|y|+1)^{2n-1}}\, dyds
\\
&\quad + \int_0^t \int_{\R^n_+} \frac1{(|x-y|+\sqrt{t-s})^n\sqrt{t-s}} \frac{\LN(y,s)}{(|y|+1)^{2n-2}(y_n+1)^{2a}}\, dyds =: I_1+I_2.
}

By the same estimate of $|B(f,g)(x,t)| $ in \eqref{0426b} as $\bka{y}^{-2n+1} \le \bka{y}^{-2n+2}$, we have
\EQ{\label{0426c}
I_1 \lec  \frac{1}{\bka{x}^n} .
}

For $I_2$, note that $|s-1|<1$ since $s\in(0,2)$.
When $a\in (0,\frac12)$, let $\be=\frac12-a \in (0,\frac12)$,
\EQN{
\frac{\LN(y,s)}{(y_n+1)^{2a}}
= \frac1{(y_n^2+|s-1|)^{\frac12-a}(y_n+1)^{2a}}
&\lec \frac1{|s-1|^{\be}}.
}

When $a=1/2$, choosing $\be=1/4$, (or any value in $(0,\frac12)$)
\EQN{
\frac{\LN(y,s)}{(y_n+1)^{2a}}
= \frac{\log\bke{2+\frac1{y_n^2+|s-1|}}}{(y_n+1)}
&\lec 1 + (y_n^2+|s-1|)^{-\be}
\lec  \frac1{|s-1|^{\be}}.
}
Hence  for all $a\in (0,1/2]$, changing variables $s \to t-s$  and using Lemma \ref{lemma2.2} and $n \ge 3$,
\EQN{
I_2
&\lec \int_0^t  \int_{\R^n_+} \frac1{(|x-y|+\sqrt{s})^n} \frac1{(|y|+1)^{2n-2} }\,dy\,\frac1{\sqrt{s}|t-s-1|^\be}\,ds\\
&\lec \int_0^t \bkt{  \frac1{(|x|+\sqrt{s}+1)^{2n-2}}\, \log \frac{|x|+\sqrt{s}+1}{\sqrt{s}} + \frac1{(|x|+\sqrt{s}+1)^{n}} } \,\frac1{\sqrt{s}|t-s-1|^\be}\, ds\\
&\lec \frac1{(|x|+1)^{2n-2}} \int_0^t \frac{ {\log(|x|+3)+ |\log \sqrt s|}}{\sqrt{s} |t-s-1|^\be}\, ds  + \frac1{(|x|+1)^n} \int_0^t \frac1{\sqrt{s} |t-s-1|^\be}\, ds.
}
Since $\be<1/2$, both time integrals converge and are uniformly bounded in $t$, and
\EQ{\label{0426d}
I_2 \lec \frac{\log (|x|+3)}{\bka{x}^{2n-2}} + \frac1{\bka{x}^n}\lec \frac1{\bka{x}^n}.
}
Combining \eqref{0426c} and \eqref{0426d}, we have derived
\[
|\na B(f,g)(x,t)| \lec \frac1{\bka{x}^n}.
\]
This completes the proof of \thref{lem4.1}.
\end{proof}

\begin{prop}\thlabel{prop4.1}
Let $n\ge3$ and $\hat{v}$ be a solenoidal vector field in $\R^n_+ \times (0,2)$ satisfying the pointwise bounds \eqref{eq-0415-a}.
There is a small constant $\al_0>0$ such that, for any $\al \in (-\al_0,\al_0)$,
there is a solution $v$ to
\EQ{\label{eq-stokes}
\pd_t v - \De v + \nb \pi = - (\al \hat v + v) \cdot \nb (\al \hat v + v), \quad \div v=0
}
 in $\R^n_+ \times (0,2)$,
with $v(x',0,t)=0$,  $v(x,0)=0$, and the pointwise bounds
\EQ{\label{prop4.1-eq1}
\bke{ |v(x,t)| +  |\na v(x,t)|} \lec \bka{x}^{-n}.
 }
\end{prop}

\begin{proof}
By the solution formula \eqref{E1.3}, a mild solution $v$ of \eqref{eq-stokes}  satisfies
\EQ{\label{prop4.1-eq2}
v = B(\al \hat v + v, \al \hat v + v) = \al^2 B(\hat v ,\hat v) + \al B(\hat v ,v) + \al B(v ,\hat v) + B(v ,v).
}
The main term of $v$ is the source term $\al^2 B(\hat v ,\hat v)$. %
We construct $v$ by iteration:
\[
v_i^{(0)} (x,t) = 0,\qquad
v_i^{(m+1)} (x,t) = B(\al \hat{v} + v^{(m)}, \al \hat{v} + v^{(m)})_i (x,t)
\]
for $m \ge 0$. Let ${X}$ and ${Y}$ be defined as in \eqref{XTYT.def} and $A=\norm{\hat v}_{{Y}}$. Note that $A \lec 1$ by Proposition \ref{th3.1}.

For the first iteration $v_i^{(1)} (x,t) = B(\al \hat{v} , \al \hat{v} )_i$, we have by \thref{lem4.1}
\[
\norm{v^{(1)}}_{{X}} = \al^2 \norm{B(\hat v,\hat v)}_{{X}} \le C_1 \al^2 A^2.
\]

Assume, for the sake of induction, that $\norm{v^{(m)}}_{{X}} \le M:=4C_1 \al^2 A^2$ for small $\al\in\R$. Then, by \thref{lem4.1},
\EQN{
\norm{v^{(m+1)}}_{{X}}
&\le C_1 \norm{\al \hat v+ v^{(m)}}_Y^2 \le  C_1 \bke{\norm{\al \hat v}_Y + \|v^{(m)}\|_X }^2
\\
&\le
C_1 (|\al| A + M)^2 \le M,
}
if $4C_1 |\al| A \le 1$ (which implies $M \le |\al| A  $).
By induction, we have $\norm{ v^{(m)} }_{{X}} \le M=O(\al^2)$ for all $m\ge1$.

We next show the convergence of the sequence $\{v^{(m)}\}_m$.
Note
\EQN{
v^{(m+1)} - v^{(m)} &=  B(\al \hat v+v^{(m)}, \al \hat v+v^{(m)})- B(\al \hat v+ v^{(m-1)},\al \hat v+v^{(m-1)})
\\ &= B(\al \hat v+v^{(m)}, v^{(m)}-v^{(m-1)}) + B( v^{(m)}-v^{(m-1)},\al \hat v+v^{(m-1)}) .
}
By \thref{lem4.1},
\EQN{
\norm{ v^{(m+1)} - v^{(m)} }_{{X}}
& \le C_1 \bke{\norm{\al \hat v}_Y + \|v^{(m)}\|_X } \cdot  \norm{v^{(m)} -v^{(m-1)}}_X
\\
& \quad+ C_1 \bke{\norm{\al \hat v}_Y + \|v^{(m-1)}\|_X } \cdot  \norm{v^{(m)} -v^{(m-1)}}_X
\\
&
\le 2C_1(|\al| A+M)\norm{v^{(m)} -v^{(m-1)}}_X
\le \frac 12 \norm{v^{(m)} -v^{(m-1)}}_X
}
if $8C_1|\al| A\le 1$.
Therefore, by contraction mapping principle, $\{v^{(m)}\}_m$ converges to a solution $v\in {X}$ of \eqref{prop4.1-eq2} with $\norm{v}_{{X}} \le M$  if $\al$ is sufficiently small,
\[
|\al| \le \al_0:=(8C_1\norm{\hat v}_{{Y}})^{-1}.
\]
This completes the proof of \thref{prop4.1}.
\end{proof}

\medskip

\begin{proof}[Proof of Theorem \ref{th1.3}]
For any given $\phi$ satisfying \eqref{boundary-data1}--\eqref{boundary-data3} for $0<a\le 1/2$,
let $\hat v(x,t)$ be defined as in Proposition \ref{th3.1} and $v$ be the solution constructed in \thref{prop4.1}. It is easy to show that the vector field $u = \al \hat v + v$ is a solution to the Navier-Stokes equations \eqref{NS} with boundary data $\al \phi(\xi',s)$, zero initial data and zero force. It satisfies the pointwise bounds \eqref{eq-0415-a} by \eqref{th31eq1}, \eqref{th31eq2} and \eqref{prop4.1-eq1}. Thus it has finite global energy \eqref{th1.1eq1}.

We now show the H\"older continuity of $u= \al \hat v + v$. We have already shown the H\"older continuity of $ \hat v$ in Theorem \ref{th1.1}. We can consider $v$ as a solution of the Stokes system with zero initial and boundary values, and nonzero force $\nb\cdot F$, $F=u\otimes u$. By the pointwise bounds $|u(x)|\lec \bka{x}^{1-n}$ from \eqref{eq-0415-a}, $F(x) \in L^p(\R^n_+\times(0,2))$ for any $p\in [1,\infty]$. By Proposition 2.3 of Chang-Choe-Kang \cite{MR3815542} and taking $p<\infty$ arbitrarily large, we get
$v \in C^{2c,c}(\overline{\R^n_+}\times [0,2])$ for any $c<1/2$. Hence $u = \al \hat v + v \in C^{2b,b}(\overline{\R^n_+}\times [0,2])$.

Now we consider the regularity of the pressure $\pi$. Recall $u = \al \hat v + v$ and decompose $\pi=\hat p + \pi_v$, where $(v,\pi_v)$ solves the nonhomogeneous Stokes system \eqref{eq-stokes}
with source term $f= - (\al \hat v + v) \cdot \nb (\al \hat v + v)$,
and zero initial and boundary values.

By Proposition \ref{th3.1} and \eqref{prop4.1-eq1},
\EQ{
|\al \hat v + v|(x,t) \lec \frac 1{\bka{x}^{n-1}},
}
and
\EQ{
|\nb (\al \hat v + v)(x,t)| \lec \frac {1+\de_{a=\frac12} \log\bke{2 + \frac 1{x_n+\sqrt{|t-1|}}}}{\bka{x}^{n-1}(x_n+1)^{2a}(x_n+\sqrt{|t-1|})^{1-2a}}.
}
Hence, when $a=1/2$,
\EQ{
|f(x,t)|\lec \frac { \log\bke{2 + \frac 1{\sqrt{|t-1|}}}}{\bka{x}^{2n-2}(x_n+1)}
\in L^r(0,2; L^1\cap L^\infty (\R^n_+)),
}
for any $r\in (1,\infty)$. When $0<a<1/2$,
\EQ{
|f(x,t)| \lec \frac {1}{\bka{x}^{2n-2}(x_n+1)^{2a}(x_n+\sqrt{|t-1|})^{1-2a}}\in L^r(0,2; L^q (\R^n_+)),
}
for any $r,q\in (1,\infty)$ with
\EQ{\label{rq.cond}
\frac 1q + \frac 2r > 1-2a.
}
By maximal regularity theorem for Stokes system  (by Solonnikov \cite{MR0415097} for equal exponents, by Sohr and vol Wahl \cite{MR847086} for mixed exponents, and by Giga and Sohr \cite{GiSo91} for $T$-independent constant),
\EQ{
\norm{\nb \pi_v}_{L^r(0,2; L^q(\R^n_+))} \le C \norm{f}_{L^r(0,2; L^q(\R^n_+))}  \le C,
}
for any $r, q \in (1,\infty)$ satisfying \eqref{rq.cond}.
By Sobolev imbedding, for $q \in (1,n)$ and $1/m=1/q-1/n$,
\EQ{\label{piv.space}
\pi_ v \in L^r(0,2; L^{m}(\R^n_+)),
}
for any $r\in (1,\infty)$ and any $m \in (\frac n{n-1},\infty)$  satisfying
\EQ{\label{piv.space2}
\frac 1m + \frac 2r > 1-2a-\frac 1n.
}

Recall from \eqref{est-p-hat} that for $n \ge 3$,
\EQ{
   |\hat p(x,t)| \lec |1-t|^{a-1}  \bka{x}^{2-n} ,  \quad  0<t<2.
}
Together with \eqref{piv.space}-\eqref{piv.space2}, the total pressure $\pi = \hat p + \pi_v$ satisfies \eqref{pi-regular}, noting that
the condition $r< \frac 1{1-a}$ implies \eqref{piv.space2}.

Finally, with the choice \eqref{gxiform} of $g$, the unboundedness of normal derivative \eqref{th1.3eq1} follows from \eqref{th1.1eq2} and \eqref{prop4.1-eq1}.
\end{proof}

\begin{remark}
The H\"older continuity of $v$ can be also proved by hand, but  it is easier to refer to Chang-Choe-Kang \cite{MR3815542}.
\end{remark}

\section{Appendix A.\ Dipole bumps}
\renewcommand{\theequation}{A.\arabic{equation}}
\renewcommand{\thethm}{A.\arabic{thm}}
\label{Sec5}

In this appendix we consider variations of the boundary data, and their corresponding solutions of the Stokes system. For simplicity, we only consider space dimension $n=3$, and the boundary data which we call \emph{dipole bumps}.

\begin{prop}[Dipole bumps]\label{thA.1}
Let $n =3$ and $\phi(\xi',s)$ be defined on $\Si \times \R$ by
\eqref{boundary-data1}-\eqref{boundary-data3} with $0<a \le 1/2$.
Let $v(x,t)$ be the solution of the Stokes system \eqref{E1.1} with boundary data $\phi$, zero initial data and zero force, given by \eqref{hatv.formula}.
Then $v$ satisfies the pointwise bounds in Proposition \ref{th3.1}, and it has finite global energy
\eqref{th1.1eq1}.
If we further choose boundary data of the form
\EQ{ \label{thA.1-eq1}
\phi(\xi',t) = G(\xi') \,h(t)\, e_3, \quad
G(\xi')=-g(\xi'-10e_1)+g(\xi'+10e_1),
}
where $g(\xi')$ is as chosen in \eqref{gxiform} and $e_1=(1,0)$, then when $|x'|>100$, we have
\EQS{\label{dipole-10}
\lim_{x_3\searrow 0}\partial_3 v_{1} (x_1, x_2, x_3,1 )=
\begin{cases}
{-\infty }&\quad \bket{|x_2|>\sqrt{2}\,(|x_1|+12) },\\[2mm]
{\infty }&\quad \bket{|x_2| <\sqrt{2}\,(|x_1|-12)},
\end{cases}
}
and
\EQS{\label{dipole-20}
\lim_{x_3\searrow 0}\partial_3 v_{2} (x_1, x_2, x_3,1 )=
\begin{cases}
{\infty }&\quad \bket{x_1x_2>0,\, |x_1|>1},\\[2mm]
{-\infty }&\quad \bket{x_1x_2<0,\, |x_1|>1}.
\end{cases}
}
\end{prop}
\emph{Comments on Proposition \ref{thA.1}:}
\EN{
\item We call the boundary data in \eqref{thA.1-eq1} \emph{dipole bumps}, as it consists of influx for $\xi' \in B'_1(-10e_1)$, and outflux for $\xi' \in B'_1(10e_1)$. In contrast, we call the boundary data chosen in Proposition \ref{th1.2} as \emph{single bumps}.

\item The main interest of the dipole bumps is the different topologies of the blow-up regions
\[
E_i^\pm = \{ x' \in \Si_{\text{ext}}:  \ \lim_{x_3\searrow 0}\partial_3 v_{i} (x_1, x_2, x_3 ) = \pm \infty\},\quad i=1,2,
\]
where $\Si_{\text{ext}} = \{ x' \in \Si:  |x'|>100\}$.
In Proposition \ref{th1.2}, $ E_i^+ = \{ x' \in \Si_{\text{ext}} :  \ x_i < 0\} $ and $ E_i^- = \{ x' \in \Si_{\text{ext}} :  \ x_i > 0\} $ for single bumps. Each of them is simply connected.
In contrast, for dipole bumps, $E_2^\pm = \{ x' \in \Si_{\text{ext}} :  \ \pm x_1x_2 > 0\} $. For $E_1^\pm$, it appears that $E_1^+ = \{ x' \in \Si_{\text{ext}} :  |x_1| > \ga(|x_2|)\} $ and $E_1^- = \{ x' \in \Si_{\text{ext}} :  |x_1| < \ga(|x_2|)\} $ for some increasing positive function $\ga$, although we are unable to prove it fully. Each of them has two components. Note that these regions are the same for all $a \in (0,\frac12)$.

\item One may also consider \emph{double bumps} which is similar to  \eqref{thA.1-eq1} but
\[
 G(\xi')=g(\xi'-10e_1)+g(\xi'+10e_1),
\]
i.e., both bumps have the same sign. We found that the topologies of $E_i^\pm$ are the same as single bumps considered in Proposition \ref{th1.2}. In fact they have the same sets $E_i^\pm$. Hence we omit the details.
}

The following two lemmas will be useful in our proof of Proposition \ref{thA.1}.

\begin{lem}\label{lem.A2}
Let $h,b,L>0$, $H(t)=(t^2+h^2)^{-\frac{1}{2}}$ and
\[
F(t,b,L) = H(t) - H(t+b) - H(t+L) + H(t+b+L),\quad (-b<t<\infty).
\]
For any $\e \in (0,1]$, we have
\EQ{\label{0528a}
\begin{cases}
F(t,b,L)> \frac14\e{bL}   H^3(t+b+L) &\quad \text{if} \quad t \ge \frac 1{\sqrt{2-\e}}\, h,\\[2mm]
F(t,b,L)< -\frac14\e{bL}   H^3(t+b+L) &\quad \text{if} \quad t+b+L< \frac 1{\sqrt{2+\e}}\, h.\\[2mm]
\end{cases}
}
\end{lem}
\begin{proof}
Note that
\EQN{
F &= -\int_t^{t+b} H'(\tau)d\tau + \int_{t+L}^{t+b+L} H'(\tau)d\tau
\\
&=\int_t^{t+b} \bke{H'(\tau+L) -H'(\tau)} d\tau
=\int_t^{t+b}\int_{\tau}^{\tau+L} H''(s) \,ds\, d\tau.
}
Also note
\[
H'(s)=-s(s^2+h^2)^{-\frac{3}{2}},
\quad
H''(s)=(s^2+h^2)^{-\frac{5}{2}}(2s^2-h^2).
\]
Hence $H''(s)>0$ if $s\ge t>\frac{h}{\sqrt{2}}$. If we further assume $s \ge \frac 1{\sqrt{2-\e}}\, h$, then
$2s^2-h^2 \ge C_\e (s^2 +h^2)$, $C_\e=\frac\e{3-\e}\ge \frac\e4$, and hence $H''(s) \ge C_\e(s^2+h^2)^{-\frac{3}{2}}$,
\[
F \ge\int_t^{t+b}\int_{\tau}^{\tau+L} C_\e(s^2 +h^2)^{-\frac{3}{2}}\,ds\, d\tau\ge C_\e {bL}   H^3(t+b+L).
\]

On the other hand, $H''(s)<0$ if $|s| < \frac{h}{\sqrt{2}}$. If we further assume $t+b+L   \le \frac 1{\sqrt{2+\e}}\, h$ and $s \in (t,t+b+L)$,
then $|s|\le \frac 1{\sqrt{2+\e}}\,h$ even if $-b<t<0$ (in which case $|t|\le b \le L \le \frac 1{\sqrt{2+\e}}\,h$). Then
$2s^2-h^2 \le -C_\e (s^2+h^2)$, $C_\e=\frac\e{3+\e}\ge \frac\e4$ and hence $H''(s) \le -C_\e(s^2+h^2)^{-\frac{3}{2}}$,
\[
F \le\int_t^{t+b}\int_{\tau}^{\tau+L}-C_\e(s^2 +h^2)^{-\frac{3}{2}}\,ds\, d\tau\le -C_\e{bL}   H^3(t+b+L).
\qedhere
\]
\end{proof}

\begin{lem}\label{lem.A3}
Let $H(x,y)=(x^2+y^2)^{-\frac{1}{2}}$, $0\le t<u$ and $0 \le a <b$. Then
\EQ{\label{0529a}
F(t,u,a,b)= H(t,a) - H(t,b) - H(u,a) + H(u,b) \ge \frac 34 (u^2-t^2)(b^2-a^2) H^5(u,b).
}
\end{lem}
\begin{proof}
Note that
\EQN{
F &= -\int_a^{b} H_y(t,y)dy + \int_a^{b} H_y(u,y)dy
\\
&=\int_a^{b} \int_t^{u} H_{xy}(x,y)\,dx\,dy .
}
Also note
\[
H_y(x,y) = - y (x^2+y^2)^{-\frac{3}{2}},
\quad
H_{xy}(x,y)=3xy (x^2+y^2)^{-\frac{5}{2}}.
\]
Hence
\[
F \ge \int_a^{b} \int_t^{u} 3xy (u^2+b^2)^{-\frac{5}{2}} \,dx\,dy =\frac 34 (u^2-t^2)(b^2-a^2) H^5(u,b).\qedhere
\]
\end{proof}

\begin{proof}[Proof of Proposition \ref{thA.1}]
Due to superposition, $ v = v_+ + v_-$, where
$v_+$ and $v_-$ are the solutions of the Stokes system for boundary data $-g(\xi'-10e_1)$ and
$g(\xi'+10e_1)$, respectively. They are translations of $\hat v$ of Proposition \ref{th1.2} with an opposite sign for $v_+$.
By Proposition \ref{th1.2}, we have
\EQ{\label{0525a}
\lim_{x_3\searrow 0}\partial_3 v_{+,1} (x,1)=
\begin{cases}
\infty &\quad \bket{x_1>10}\cap \bket{|x'|>100},\\[2mm]
-\infty &\quad \bket{x_1<10}\cap \bket{|x'|>100},
\end{cases}
}
\EQ{\label{0525b}
\lim_{x_3\searrow 0}\partial_3 v_{+,2} (x,1)=
\begin{cases}
\infty &\quad \bket{x_2>0}\cap \bket{|x'|>100},\\[2mm]
-\infty &\quad \bket{x_2<0}\cap \bket{|x'|>100},
\end{cases}
}
and
\EQ{\label{0525c}
\lim_{x_3\searrow 0}\partial_3 v_{-,1} (x,1)=
\begin{cases}
-\infty &\quad \bket{x_1>-10}\cap \bket{|x'|>100},\\[2mm]
\infty &\quad \bket{x_1<-10}\cap \bket{|x'|>100},
\end{cases}
}
\EQ{\label{0525d}
\lim_{x_3\searrow 0}\partial_3 v_{-,2} (x,1)=
\begin{cases}
-\infty &\quad \bket{x_2>0}\cap \bket{|x'|>100},\\[2mm]
\infty &\quad \bket{x_2<0}\cap \bket{|x'|>100}.
\end{cases}
}

Also note that, since our choice of $g(\xi_1,\xi_2)$ is even in both $\xi_1$ and $\xi_2$, our solution $ v = v_+ + v_-$ is such that $v_1$ is even in both $x_1$ and $x_2$, while $v_2$ is odd in both $x_1$ and $x_2$, by the symmetry between $v_+$ and $v_-$. Hence for our proof below, it suffices to consider $x_1\ge0$ and $x_2\ge 0$.

We first consider $\pd_3 v_1(x,1)$, the normal derivative of the first component.
Its main term is $\pd_3 I_2(x,1)$ in \eqref{eq3.45} with $i=1$. Hence
\EQN{
&\partial_{x_3}v_1 \approx \int_{0}^{1} \int_{\Si}   e^{-\frac{x_3^2}{4s}}\bke{\frac{1}{2}-\frac{x_3^2}{4 s}}
\frac{1}{s^{\frac{5}{2}}}\, K(x'-\xi',s)\, \pd_1[-g(\xi'-10e_1)+g(\xi'+10e_1)] s^a\, d\xi' ds
\\
&=\int_{0}^{1} \int_{\Si}   e^{-\frac{x_3^2}{4s}}s^{a -\frac 52}\bke{\frac{1}{2}-\frac{x_3^2}{4 s}}
   \bkt{K(x'-\xi'+10e_1,s) - K(x'-\xi'-10e_1,s)}
 \pd_1g (\xi') \, d\xi' ds
\\
&  = \int_{0}^1 \int_{\Si, \xi_1>0}  e^{-\frac{x_3^2}{4s}}s^{a -\frac 52}\bke{\frac{1}{2}-\frac{x_3^2}{4 s}}
\\
& \hspace{35mm} \cdot   \mat{-K(x'-\xi'+10e_1,s) + K(x'-\xi'-10e_1,s) \\[1mm] +K(x'-\xi'_*+10e_1,s) - K(x'-\xi'_*-10e_1,s)  }
 |\pd_1g(\xi') |\, d\xi' ds.
}
Above $\xi'_*=(-\xi_1,\xi_2)$, and we have used
$\pd_1 g(\xi') = -\sgn(\xi_1) |\pd_1 g(\xi')| = -\sgn(\xi_1) |\pd_1 g(\xi'_*)|$.

In the case that $-10<x_1<10$ and $|x'|>100$, we have
$\lim_{x_3\to 0}\partial_3 v_{1} (x,1)=-\infty$
by summing \eqref{0525a} and \eqref{0525c}. To show the first part of \eqref{dipole-10}, if suffices to assume $x_1>10$,  $|x'|>100$ and $x_2>\sqrt{2}(x_1+12)$. As this is a strict inequality, we can choose $\e\in (0,1)$ such that
$x_2>\sqrt{2+\e}(x_1+12)$.

Denote
\EQS{\label{d+def}
&d_+=|x'+10e_1-\xi'|, \quad d_-=|x'-10e_1-\xi'|,   \\
&  d_+^{*}=|x'+10e_1-\xi'_{*}|,\quad d_-^{*}=|x'-10e_1-\xi'_{*}|.
}
We have
\EQ{\label{d+order}
d_-<d_-^*<d_+<d_+^*, \quad \text{for}\quad  0<\xi_1<1, \quad 10<x_1.
}
Note that $\frac{1}{2}-\frac{x_3^2}{4 s}<0$ if and only if $0<s<x^2_3/2$. By splitting the time integral as $\int_0^{x_3^2/2}ds + \int_{x_3^2/2} ^1ds$ and using Lemma \ref{th3.2},
\EQN{
\partial_{x_3}v_1
&\lec \int_{0}^{x_3^2/2} \int_{B_1', \xi_1>0}  e^{-\frac{x_3^2}{4s}}s^{a -\frac 52} \bke{\frac{x_3^2}{4 s}-\frac{1}{2}}%
\\
&\hspace{10mm} \cdot
\bkt{ \frac {ms}{(m-1)d_+}-\frac {ms}{(m+1)d_-}-\frac {ms}{(m+1)d_+^*}+\frac {ms}{(m-1)d_-^*}}
|\pd_1g(\xi') |\, d\xi' ds
\\
& +  \int_{x_3^2/2}^1 \int_{B_1', \xi_1>0}  e^{-\frac{x_3^2}{4s}}s^{a -\frac 52}\bke{\frac{1}{2}-\frac{x_3^2}{4 s}} %
\\
&\hspace{10mm} \cdot
\bkt{ -\frac {ms}{(m+1)d_+}+\frac {ms}{(m-1)d_-}+\frac {ms}{(m-1)d_+^*}-\frac {ms}{(m+1)d_-^*}}
|\pd_1g(\xi') |\, d\xi' ds + \lot,
}
where $\lot$ means lower order terms.
Changing variables $\si = x_3^2/4s$ as in \eqref{eq-0531-a} and following the same computation as in \eqref{eq-0622-a} for $a<1/2$ and in \eqref{eq-0622-b} for $a=1/2$ with $K_1,K_2$ being replaced by $P_1,P_2$, we get
\[
\partial_{x_3}v_1\lec \one_{a<\frac12}\, x_3^{{2a-1}}
\bke{M_2 P_2 + M_1 P_1} + \one_{a=\frac12}\, P_1 \log \frac2{x_3}  + \lot,
\]
where $M_1$ and $M_2$ are defined in \eqref{M1M2def}, $0<M_2<M_1$, and
\EQN{
P_2&=\int_{B'_1 ,\xi_1>0}
\bke{-\frac{m}{(m+1){d_-}} +\frac{m}{(m-1){d_-^*}} +\frac{m}{(m-1){d_+}}-\frac{m}{(m+1){d_+^*}}} |\pd_1g (\xi')| \, d\xi',
\\
P_1 &=\int_{B'_1 ,\xi_1>0}
\bke{\frac{m}{(m-1){d_-}}-\frac{m}{(m+1){d_-^*}} -\frac{m}{(m +1){d_+}}+\frac{m}{(m -1){d_+^*}}} |\pd_1g (\xi')| \, d\xi'.
}
By taking $m\ge C_m|x'|^2/\e$ with $C_m$ sufficiently large,  $-P_2$ and $P_1$ are close to
\[
I:= \int_{B'_1 ,\xi_1>0}
\bke{\frac 1{ {d_-}}-\frac 1{ {d_-^*}} -\frac 1{{d_+}}+\frac 1{{d_+^*}}} |\pd_1g (\xi')| \, d\xi',
\]
in the sense that
\[
|P_2 +I| + |P_1-I| \le P_3:=\int_{B'_1 ,\xi_1>0} \frac{C\e}{C_m |x'|^3}|\pd_1g (\xi')| \, d\xi'.
\]
Therefore,
\EQ{\label{eqA11}
\partial_{x_3}v_1\lec  \one_{a<\frac12}\, x_3^{{2a-1}} [I (M_1-M_2) + CP_3] + \one_{a=\frac12}\, \log\frac2{x_3}\, [I + C P_3]+ \lot.
}
Recall \eqref{d+def} and \eqref{d+order}.
We will apply Lemma \ref{lem.A2} with $t=x_1-10-\xi_1$, $b=2\xi_1$, $L=20$ and $h=|x_2-\xi_2|$.
Observe that when $|\xi'|<1$,
\[
t>-b, \quad
t+b+L=
x_1+10+\xi_1<\frac1{\sqrt{2+\e}} h  = \frac1{\sqrt{2+\e}} |x_2-\xi_2|
\]
by our assumption that $x_1>10$ and $x_2>\sqrt{2+\e}(x_1+12)$ for some $\e\in (0,1)$. (Note that $-b<t<0$ is allowed.)
By Lemma \ref{lem.A2},
\[
{\frac 1{ {d_-}}-\frac 1{ {d_-^*}} -\frac 1{{d_+}}+\frac 1{{d_+^*}}}< -\frac {C \e\xi_1 }{(d_+^*)^3} \le - \frac {C \e}{|x'|^3} .
\]
In the last inequality we've used \eqref{gxiform} that
$\xi_1 \ge \frac 1{2\sqrt 2}$ in the support of $\pd_1 g(\xi')$ if $\xi_1>0$.  As $M_2<M_1$, we get
\[
I (M_1-M_2) + CP_3 \le \int_{B'_1 ,\xi_1>0} \bke{- \frac{C\e(M_1-M_2)}{ |x'|^3} + \frac{C\e}{C_m |x'|^3}}|\pd_1g (\xi')| \, d\xi'<0,
\]
and
\[
I + CP_3 \le \int_{B'_1 ,\xi_1>0} \bke{- \frac{C\e}{ |x'|^3} + \frac{C\e}{C_m |x'|^3}}|\pd_1g (\xi')| \, d\xi'<0,
\]
if $C_m$ is sufficiently large.
By \eqref{eqA11}, $\partial_{x_3}v_1\le -C \one_{a<\frac12}\, x_3^{{2a-1}} -C \one_{a=\frac12}\, \log(2/x_3)+ \lot$ and hence
 $\lim_{x_3\to0}\partial_{x_3}v_1=-\infty$, showing the first half of \eqref{dipole-10}.

We next show $\lim_{x_3\to0}\partial_{x_3}v_1=+\infty$ when $|x'|>100$ and $0\le x_2<\sqrt{2}(x_1-12)$ for some $\e\in (0,1)$. As this is a strict inequality, we can choose $\e\in (0,1)$ such that
$x_2<\sqrt{2-\e}(x_1-12)$. Instead of an upper bound, we now want a lower bound. Similar estimates leading to \eqref{eqA11} gives
\EQ{\label{eqA12}
\partial_{x_3}v_1\gec  \one_{a<\frac12}\, x_3^{{2a-1}} [I (M_1-M_2) - CP_3] + \one_{a=\frac12}\, \log\frac2{x_3} [I - CP_3]- \lot.
}
When $|\xi'|<1$, we have
\[
t=x_1-10-\xi_1 \ge \frac1{\sqrt{2-\e}}\, h = \frac1{\sqrt{2-\e}}\, |x_2-\xi_2|
\]
by our assumption $x_2<\sqrt{2-\e}(x_1-12)$ and $|\xi_i|<1$.
By Lemma \ref{lem.A2},
\[
{\frac 1{ {d_-}}-\frac 1{ {d_-^*}} -\frac 1{{d_+}}+\frac 1{{d_+^*}}}> \frac {C \e \xi_1 }{(d_+^*)^3} \ge \frac {C \e}{|x'|^3}
\]
using \eqref{gxiform} that
$\xi_1 \ge \frac 1{2\sqrt 2}$ in the support of $\pd_1 g(\xi')$ if $\xi_1>0$.  As $M_2<M_1$, we get
\[
I (M_1-M_2) - CP_3 \ge \int_{B'_1 ,\xi_1>0} \bke{ \frac{C\e(M_1-M_2)}{ |x'|^3} - \frac{C\e}{C_m |x'|^3}}|\pd_1g (\xi')| \, d\xi'>0,
\]
and
\[
I - CP_3 \ge \int_{B'_1 ,\xi_1>0} \bke{ \frac{C\e}{ |x'|^3} - \frac{C\e}{C_m |x'|^3}}|\pd_1g (\xi')| \, d\xi'>0,
\]
if $C_m$ is sufficiently large.
By \eqref{eqA12}, $\partial_{x_3}v_1\ge C(x')\one_{a<\frac12}\, x_3^{{2a-1}} + C(x') \one_{a=\frac12}\, \log(2/x_3) + \lot$ and hence
 $\lim_{x_3\to0}\partial_{x_3}v_1=+\infty$, showing the second half of \eqref{dipole-10}.

\bigskip

We next consider $\pd_3 v_2(x,1)$, the normal derivative of the second component. Its main term is $\pd_3 I_2(x,1)$ in \eqref{eq3.45} with $i=2$.  It is computed as follows:
\EQN{
&\partial_{x_3}v_2 \approx \int_{0}^{1} \int_{\Si}   e^{-\frac{x_3^2}{4s}}\bke{\frac{1}{2}-\frac{x_3^2}{4 s}}
\frac{1}{s^{\frac{5}{2}}}\, K(x'-\xi',s)\, \pd_2[-g(\xi'-10e_1)+g(\xi'+10e_1)] s^a\, d\xi' ds
\\
&=\int_{0}^{1} \int_{\Si}   e^{-\frac{x_3^2}{4s}}s^{a -\frac 52}\bke{\frac{1}{2}-\frac{x_3^2}{4 s}}
   \bkt{K(x'-\xi'+10e_1,s) - K(x'-\xi'-10e_1,s)}
 \pd_2g (\xi') \, d\xi' ds
\\
&  = \int_{0}^1 \int_{\Si, \xi_2>0}  e^{-\frac{x_3^2}{4s}}s^{a -\frac 52}\bke{\frac{1}{2}-\frac{x_3^2}{4 s}}
\\
& \hspace{35mm} \cdot   \mat{-K(x'-\xi'+10e_1,s) + K(x'-\xi'-10e_1,s) \\[1mm] +K(x'-\xi'_\sharp+10e_1,s) - K(x'-\xi'_\sharp-10e_1,s)  }
 |\pd_2g(\xi') |\, d\xi' ds.
}
Above $\xi'_\sharp=(\xi_1,-\xi_2)$. Denote
\EQS{\label{dt+def}
&d_+=|x'+10e_1-\xi'|, \quad d_-=|x'-10e_1-\xi'|,   \\
&  d_+^{\sharp}=|x'+10e_1-\xi'_{\sharp}|,\quad d_-^{\sharp}=|x'-10e_1-\xi'_{\sharp}|.
}

Assume $x_1>1$, $x_2>0$, and $0<\xi_2<1$. Then
\EQ{\label{d+order-v2}
d_-<\min(d_-^\sharp,d_+)<\max(d_-^\sharp,d_+)<d_+^\sharp.
}
Following the argument for $\pd_{x_3}v_1$, we have
\EQ{\label{eqA16}
\partial_{x_3}v_2\gec \one_{a<\frac12}\, x_3^{{2a-1}} [I_\sharp (M_1-M_2) - CP_4] + \one_{a=\frac12}\, \log\frac2{x_3} [I_\sharp - CP_4]- \lot,
}
where
\EQN{
I_\sharp&= \int_{B'_1 ,\xi_2>0}
\bke{\frac 1{ {d_-}}-\frac 1{ {d_-^\sharp}} -\frac 1{{d_+}}+\frac 1{{d_+^\sharp}}} |\pd_2g (\xi')| \, d\xi',
\\
P_4&=\int_{B'_1 ,\xi_2>0} \frac{Cx_2}{C_m |x'|^5}|\pd_2g (\xi')| \, d\xi',
}
by taking $m \ge C_m |x'|^4/x_2$.
We now apply Lemma \ref{lem.A3} with
$$
t=|x_1-10-\xi_1|, \quad u=|x_1+10-\xi_1|, \quad
 a=|x_2-\xi_2|, \quad b=|x_2+\xi_2|,
$$
to get
\[
{\frac 1{ {d_-}}-\frac 1{ {d_-^\sharp}} -\frac 1{{d_+}}+\frac 1{{d_+^\sharp}}}\ge \frac {C (u^2-t^2)(b^2-a^2)}{(d_+^\sharp)^5}
= \frac {C (x_1-\xi_1)x_2 \xi_2}{(d_+^\sharp)^5}
\ge \frac {C x_2 }{|x'|^5} ,
\]
using \eqref{gxiform} that
$\xi_2 \ge \frac 1{2\sqrt 2}$ and $|\xi_i| \le \frac 4{5\sqrt 2}$
in the support of $\pd_2 g(\xi')$ if $\xi_2>0$.  We have also used $x_1>1$. As $M_2<M_1$, we get
\[
I_\sharp (M_1-M_2) - CP_4
 \ge \int_{B'_1 ,\xi_2>0} \bke{ \frac{Cx_2(M_1-M_2)}{ |x'|^5} - \frac{Cx_2}{C_m |x'|^5}}|\pd_2g (\xi')| \, d\xi'>0,
\]
if $C_m$ is sufficiently large.
By \eqref{eqA16}, $\partial_{x_3}v_2\ge C(x') \one_{a<\frac12}\, x_3^{{2a-1}} + C(x') \one_{a=\frac12}\, \log(2/x_3) + \lot$ and hence
 $\lim_{x_3\to0}\partial_{x_3}v_2=+\infty$, which gives  \eqref{dipole-20}.
\end{proof}

\section{Appendix B.\ Estimates of $C_i$}
\renewcommand{\theequation}{B.\arabic{equation}}
\renewcommand{\thethm}{B.\arabic{thm}}
\label{Sec6}

In this appendix we prove Lemmas \ref{lem2.5} and \ref{Ci_estimate},
 derivative formulas and estimates of the function $C_i(x,y,t)$ defined in \eqref{eq_def_Ci} and its derivatives.

\begin{proof}[The proof of Lemma \ref{lem2.5}]
We first prove \eqref{eq0108c} {and \eqref{eq0108a}. By} definition \eqref{eq_def_Ci} of $C_i$,
\EQN{
\pd_{x_n}C_i(x,y,t)=&~\pd_{y_n}C_i(x,y,t)+\int_\Si\pd_n\Ga(z',y_n,t)\pd_iE(x'-y'-z',x_n)\,dz'\\
=&~\pd_{y_n}C_i(x,y,t)+\left(\pd_{y_n}e^{-\frac{y_n^2}{4t}}\right)\pd_{x_i}\int_\Si\Ga(z',0,t)E(x'-y'-z',x_n)\,dz'\\
=&~\pd_{y_n}C_i(x,y,t)+\left(\pd_{y_n}e^{-\frac{y_n^2}{4t}}\right)\pd_iA(x-y',t),
}
from which one obtains \eqref{eq0108c}. On the other hand, after changing variables $C_i$ becomes
\[C_i(x,y,t)=\int_{y_n}^{x_n+y_n}\int_\Si\pd_n\Ga(z,t)\, \pd_iE((x-y^*)-z)\,dz'dz_n.\]
For $i<n$ we have \[C_i(x,y,t)=\pd_{x_i}\int_{y_n}^{x_n+y_n}\int_\Si\pd_n\Ga(z,t)\,E((x-y^*)-z)\,dz'dz_n.\]
Hence
\[\begin{aligned}\pd_{x_n}C_i(x,y,t)=&~\pd_{x_i}\int_\Si\pd_n\Ga(z',x_n+y_n,t)\, E(x'-y'-z',0)\,dz'\\
&+\pd_{x_i}\int_{y_n}^{x_n+y_n}\int_\Si\pd_n\Ga(z,t)\, \pd_nE((x-y^*)-z)\,dz'dz_n\\
=&~\pd_{x_i}\pd_{x_n}\int_\Si\Ga(z',x_n+y_n,t)\,E(x'-y'-z',0)\,dz'+\pd_{x_i}C_n(x,y,t)\\
=&~\pd_{x_i}\pd_{x_n}B(x-y^*,t)+\pd_{x_i}C_n(x,y,t),\end{aligned}
\]
which proves \eqref{eq0108a}.  We now proceed to prove the identity \eqref{eq0108b}. To do this, we first move normal derivatives in the definition \eqref{eq_def_Ci} of $C_n$ to tangential derivatives. Observe that
\[\begin{aligned}C_n(x,y,t)=&~\lim_{\varepsilon\to0_+}\int_{\varepsilon}^{x_n}\int_\Si\pd_n\Ga((x-y^*)-z,t)\pd_nE(z)\,dz'dz_n
\\=&~\lim_{\varepsilon\to0_+}\left[-\int_{\varepsilon}^{x_n}\pd_{z_n}e^{-\frac{(x_n+y_n-z_n)^2}{4t}}\int_\Si\Ga(x'-y'-z',0,t)\pd_nE(z)\,dz'dz_n\right]
\\=&~\lim_{\varepsilon\to0_+}\left[-\,e^{-\frac{y_n^2}{4t}}\int_\Si\Ga(x'-y'-z',0,t)\pd_nE(z',x_n)\,dz'\right.
\\&~~~~~~~~+e^{-\frac{(x_n+y_n-\varepsilon)^2}{4t}}\int_\Si\Ga(x'-y'-z',0,t)\pd_nE(z',\varepsilon)\,dz'\\&~~~~~~~~\left.
+\int_{\varepsilon}^{x_n}e^{-\frac{(x_n+y_n-z_n)^2}{4t}}\int_\Si\Ga(x'-y'-z',0,t)\pd_n^2E(z)\,dz'dz_n\right]
\end{aligned}\]
by integration by parts in the $z_n$-variable. Using the fact that $-\De E=\de$, we obtain
\[\begin{aligned}C_n(x,y,t)=&-e^{-\frac{y_n^2}{4t}}\pd_nA(x'-y',x_n,t)+e^{-\frac{(x_n+y_n)^2}{4t}}\pd_nA(x'-y',0_+,t)\\
&-\sum_{\beta=1}^{n-1}\lim_{\varepsilon\to0_+}\int_{\varepsilon}^{x_n}\int_\Si\Ga((x-y^*)-z,t)\pd_\beta^2E(z)\,dz'dz_n.\end{aligned}\]
Integrating by parts in the $z'$-variable, we get
\[\begin{aligned}C_n(x,y,t)=&-e^{-\frac{y_n^2}{4t}}\pd_nA(x'-y',x_n,t)+e^{-\frac{(x_n+y_n)^2}{4t}}\pd_nA(x'-y',0_+,t)\\
&-\sum_{\beta=1}^{n-1}\int_0^{x_n}\int_\Si\pd_\beta\Ga((x-y^*)-z,t)\pd_\beta E(z)\,dz'dz_n.\end{aligned}\]
Now, we compute the term $\pd_nA(x'-y',0_+,t)$. {Note} \[\begin{aligned}\pd_nA(x'-y',0_+,t)=&~\lim_{\varepsilon\to0_+}\pd_nA(x'-y',\varepsilon,t)\\
=&~\lim_{\varepsilon\to0_+}\int_\Si\Ga(x'-y'-z',0,t)\pd_nE(z',\varepsilon)\,dz'\\
=&\lim_{\varepsilon\to0_+}-\frac1{n|B_1|}\int_\Si\Ga(x'-y'-z',0,t)\underbrace{\frac\varepsilon{(|z'|^2+\varepsilon^2)^{\frac{n}2}}}_{=:f_\varepsilon(z')}\,dz'\end{aligned}\]
where $f_\varepsilon(z')=\varepsilon^{1-n}f_1(z'/\varepsilon)$ so that for all test function $\varphi$,
\[
\begin{aligned}\int_\Si\varphi(z')f_\varepsilon(z')\,dz'=&~\int_\Si\varphi(z')\varepsilon^{1-n}f_1\left(\frac{z'}\varepsilon\right)dz'=\int_\Si\varphi(\varepsilon\xi')f_1(\xi')\,d\xi'\\\overset{\varepsilon\to0_+}{\longrightarrow}&~\varphi(0)\int_\Si\frac1{(|\xi'|^2+1)^{\frac{n}2}}\,d\xi'
= \varphi(0)\cdot |\pd B^{(n-1)}_1|\int_0^\infty\frac1{(r^2+1)^{\frac{n}2}}\,r^{n-2}\,dr.\end{aligned}
\]
{Also note that} $\int_0^\infty\frac1{(r^2+1)^{\frac{n}2}}\,r^{n-2}\,dr=\frac{\sqrt{\pi}\,\Ga\left(\frac{n-1}2\right)}{2\Ga\left(\frac{n}2\right)}$,
where $\Ga$ is the Gamma function.
Thus, using $|B_1^{(n)}|=\pi^{n/2} / \Ga(\frac n2+1)$,
\begin{align}\label{pdnAlim}
\pd_nA(x'-y',0_+,t)&=-c_1 \Ga(x'-y',0,t),\\
c_1 &= \frac1{n|B_1^{(n)}|}\cdot (n-1)| B^{(n-1)}_1| \cdot \frac{\sqrt{\pi}\,\Ga\left(\frac{n-1}2\right)}{2\Ga\left(\frac{n}2\right)} = \frac12.\notag
\end{align}
Alternatively, note that $P_0(x,z')=-2\pd_nE(x'-z',x_n)$ is th Poisson kernel for the Laplace equation in $\R^n_+$. Hence
\[
U(x)=
\int \varphi(x'-z')\pd_n E(z',x_n)dz' = -\frac 12\int \varphi(z') P_0(x,z') \,dz'
\]
solves
$\De U=0$ and $\lim_{x_n\to 0_+} U(x',x_n)= -\frac12\varphi(x')$.
Replacing $\varphi(x')$ by $ \Ga(x'-y',0,t)$ we get \eqref{pdnAlim}.

Therefore, using $e^{-\frac{(x_n+y_n)^2}{4t}}\Ga(x'-y',0,t)=\Ga(x-y^*,t)$,
\EQS{
C_n(x,y,t)
=&-e^{-\frac{y_n^2}{4t}}\pd_nA(x'-y',x_n,t)-\frac12\Ga(x-y^*,t)\\
&-\sum_{\beta=1}^{n-1}\pd_{x_\beta}\int_0^{x_n}\int_\Si\Ga((x-y^*)-z,t)\pd_\beta E(z)\,dz'dz_n.}
In this form we have moved normal derivatives in the definition \eqref{eq_def_Ci} of $C_n$ to tangential derivatives.
Consequently,
\[\begin{aligned}
\pd_{x_n}C_n(x,y,t)=&-e^{-\frac{y_n^2}{4t}}\pd_n^2A(x'-y',x_n,t)- \frac12 \pd_n\Ga(x-y^*,t)\\
&-\sum_{\beta=1}^{n-1}\pd_{x_\beta}\int_\Si\Ga(x'-y'-z',y_n,t)\pd_\beta E(z',x_n)\,dz'\\
&-\sum_{\beta=1}^{n-1}\pd_{x_\beta}\int_0^{x_n}\int_\Si\pd_n\Ga((x-y^*)-z,t)\pd_\beta E(z)\,dz'dz_n\\
=&-e^{-\frac{y_n^2}{4t}}\pd_n^2A(x'-y',x_n,t)-\frac12\,\pd_n\Ga(x-y^*,t)\\
&-e^{-\frac{y_n^2}{4t}}\sum_{i=\beta}^{n-1}\pd_{x_\beta}^2A(x'-y',x_n,t)-\sum_{\beta=1}^{n-1}\pd_{x_\beta}C_\beta(x,y,t).
\end{aligned}\]
Observe that $\De_x A(x'-y',x_n,t)=0$ since $\De E=0$ for $x_n>0$. Hence the first term cancels the third term. This proves \eqref{eq0108b}.
\end{proof}

\begin{proof}[The proof of \thref{Ci_estimate}]
Note that
\begin{equation}\label{C_invariant}C_i(x,y,t)=\frac1{t^{\frac{n}2}}\,C_i\left(\frac{x}{\sqrt{t}},\frac{y}{\sqrt{t}},1\right).\end{equation}
We will first estimate spatial derivatives assuming t = 1.

\medskip

\noindent{\bf $\bullet\,\boldsymbol{\pd_{x',y'}, \pd_{y_n}}$-estimate:}
Changing the variables $w=x-y^*-z$ after taking derivatives,
\[
\pd_{x',y'}^l\pd_{y_n}^qC_i(x,y,1) = \int _\Pi\pd_{w'}^l\pd_n^{q+1}\Ga(w,1)\,\pd_iE(x-y^*-w)\,dw
\]
up to a sign,
where $\Pi=\{w\in\R^n:y_n\le w_n\le x_n+y_n\}$. It is bounded for finite $|x-y^*|$, and to prove the estimate, we may assume $|x-y^*|>100$. Decompose $\Pi=\Pi_1+\Pi_2$ where
\[\Pi_1=\Pi\cap\left\{|w|<\tfrac34\,|x-y^*|\right\},\ \ \ \Pi_2=\Pi\setminus\Pi_1.\]
Integrating by parts in $\Pi_1$ with respect to $w'$ iteratively, we have
\begin{align*}
\pd_{x',y'}^l\pd_{y_n}^q  C_i(x,y,1)
&\sim \int_{\Pi_1}\left(\pd_{w_n}^{q+1}e^{-\frac{w^2}4}\right)\pd_{w'}^l\pd_iE(x-y^*-w)\,dw\\
&+\sum_{p=0}^{l-1}\int_{\Pi\cap\left\{|w|=\frac34\,|x-y^*|\right\}}\left(\pd_{w'}^{l-1-p}\pd_{w_n}^{q+1}e^{-\frac{w^2}4}\right)\pd_{w'}^p\pd_iE\cdot\chi_{p}(x-y^*-w)\,dS_w\\
&+\int_{\Pi_2}\bke{\pd_{w'}^l\pd_{w_n}^{q+1}e^{-\frac{w^2}4}}\,\pd_iE(x-y^*-w)\,dw =I_1+I_2+I_3,
\end{align*}
where $\chi_{p}$ are bounded functions on the boundary. Strictly speaking, $p$ in $\chi_{p}$ should a multi-index. %
For $I_1$, we have
\[
|I_1| \lec \int_{\Pi_1} e^{-\frac {w^2}{5}} \frac1{|x-y^*|^{l+n-1}}\,dw \lec \frac1{|x-y^*|^{l+n-1}}.
\]
If $y_n>0$, the exponential in the integrand $e^{-\frac {w^2}{5}} \le e^{-\frac {w^2}{10}} e^{-\frac {y_n^2}{10}}$. Hence
\[
|I_1|  \lec \frac{e^{-\frac {y_n^2}{10}}} {|x-y^*|^{l+n-1}}, \quad \text{if } y_n>0.
\]

For $I_3$,
since $|w|\ge\frac34|x-y^*|$ implies that $\frac1{10}|x-y^*|^2+\frac1{100}|x-y^*-w|^2 \le \frac15 |w|^2$,
\EQN{
|I_3| &\lesssim\int_{|w|>\frac34|x-y^*|}|w|^{l+q+1}\,e^{-\frac{w^2}4}\,\frac1{|x-y^*-w|^{n-1}}\,dw\\
&\lec e^{-\frac{|x-y^*|^2}{10}}\int_{|w|>\frac34|x-y^*|}e^{-\frac{|x-y^*-w|^2}{100}}\,\frac1{|x-y^*-w|^{n-1}}\,dw
\lec e^{-\frac{|x-y^*|^2}{10}}.
}

$I_2$ is controlled for large $|x-y^*|$ since
\begin{align*}
|I_2|\lesssim&~\sum_{p=0}^{l-1}\int_{|w|=\frac34|x-y^*|}|w|^{l+q-p}\,e^{-\frac{w^2}4}\,\frac1{|x-y^*-w|^{n+p-1}}\,dS_w\\
\lesssim&~e^{-\frac{9|x-y^*|^2}{64}}\sum_{p=0}^{l-1}|x-y^*|^{l+q-p}\frac1{|x-y^*|^{n+p-1}}\int_{|w|=\frac34|x-y^*|}dS_w\\
\sim&~e^{-\frac{9|x-y^*|^2}{64}}\sum_{p=0}^{l-1} |x-y^*|^{l+q-2p}\lesssim e^{-\frac{|x-y^*|^2}{10}}.
\end{align*}

Therefore, we conclude that
\EQ{
|\pd_{x',y'}^l\pd_{y_n}^q C_i(x,y,1)|\lesssim\frac{1}{\bka{x-y^*}^{l+n-1}}\,e^{-\frac{ ((y_n)_+)^2}{10}}.
}
The factor $e^{-\frac{ ((y_n)_+)^2}{10}}=1$ if $y_n \le 0$.

\medskip

\noindent{\bf $\bullet\,\boldsymbol{\pd_{x_n}}$-estimate:} We shall prove by induction with respect to $k$ that
\begin{equation}\label{eq_pdxnCi_induction}|\pd_{x',y'}^l\pd_{x_n}^k\pd_{y_n}^q  C_i(x,y,1)|\lesssim\frac1{\bka{x-y^*}^{l+n-1}\bka{x_n+y_n}^k}\,e^{-\frac{ ((y_n)_+)^2}{20}}.
\end{equation}
For $k=0$ this estimate has already been proved. Assume that the estimate has been proved for $k-1$ and arbitrary $l$. For $i<n$, \eqref{eq0108a} and the estimate \eqref{eq_estB2} give
\begin{align*}\left|\pd_{x',y'}^l\pd_{x_n}^k\pd_{y_n}^q C_i(x,y,1)\right|=&~\left|\pd_{x',y'}^l\pd_{x_n}^{k-1}\pd_{y_n}^q \left(\pd_i\pd_nB(x-y^*,1)+\pd_{x_i}C_n(x,y,1)\right)\right|\\
\le&~\left|\pd_{x',y'}^{l+1}\pd_n^{k+q} B(x-y^*,1)\right|+\left|\pd_{x',y'}^{l+1}\pd_{x_n}^{k-1}\pd_{y_n}^q C_n(x,y,1)\right|\\
\lesssim&~\frac1{\bka{x-y^*}^{l+n-1}}\,e^{-\frac{(x_n+y_n)^2}{10}} +\frac1{\bka{x-y^*}^{l+n}\bka{x_n+y_n}^{k-1}}
\,e^{-\frac{ ((y_n)_+)^2}{20}}
\\
\lesssim&~\frac1{\bka{x-y^*}^{l+n-1}\bka{x_n+y_n}^{k}}\,e^{-\frac{ ((y_n)_+)^2}{20}}.
\end{align*}
For $i=n$, \eqref{eq0108b} and the estimate
\begin{equation*}
|\pd_x^l\pd_t^m\Ga(x,t)|\lesssim\frac1{\left(x^2+t\right)^{\frac{l+n}2+m}}
\end{equation*}
imply that
\[\begin{aligned}\left|\pd_{x',y'}^l\pd_{x_n}^k\pd_{y_n}^qC_n(x,y,1)\right|=&~\left|\pd_{x',y'}^l\pd_{x_n}^{k-1}\pd_{y_n}^q\left(-\frac12\,\pd_n\Ga(x-y^*,1) -\sum_{\beta=1}^{n-1}\pd_{x_\beta}C_\beta(x,y,1)\right)\right|\\
\lesssim&~\left|\pd_x^{l+k+q}\Ga(x-y^*,1)\right| + \sum_{\beta=1}^{n-1}\left|\pd_{x',y'}^{l+1}\pd_{x_n}^{k-1}\pd_{y_n}^qC_\beta(x,y,1)\right|\\
\lesssim&~\frac1{\bka{x-y^*}^{l+n-1}}\,e^{-\frac{ (x_n+y_n)^2}{10}}+\frac1{\bka{x-y^*}^{l+n}\bka{x_n+y_n}^{k-1}}\,e^{-\frac{ ((y_n)_+)^2}{20}}\\
\lesssim&~\frac1{\bka{x-y^*}^{l+n-1}\bka{x_n+y_n}^{k}}\,e^{-\frac{ ((y_n)_+)^2}{20}}.\end{aligned}\]
So \eqref{eq_pdxnCi_induction} holds true.
\medskip

Finally, \eqref{eq_Ci_estimate_yn} can be obtained by differentiating \eqref{C_invariant} in $t$, using \eqref{eq_pdxnCi_induction}, and applying induction.
This completes the proof.
\end{proof}

\section*{Acknowledgments}
The research of KK was partially supported by NRF-2019R1A2C1084685.
The research of BL was partially supported by NSFC-11971148.
The research of both CL and TT was partially supported by the NSERC grant RGPIN-2018-04137.

\def\cprime{$'$}

\end{document}